\documentclass{elsarticle}
\usepackage{tikz}
\usetikzlibrary{trees}
\usepackage{comment}
\usepackage{amssymb}
\usepackage{amsmath}
\usepackage{amsthm}
\usepackage{mathtools}
\usepackage{hyperref}
\newtheorem{thm}{Theorem}[section]
\newtheorem{defn}[thm]{Definition}
\newtheorem{lemma}[thm]{Lemma}
\newtheorem{prop}[thm]{Proposition}
\newtheorem{note}[thm]{Note}
\newtheorem{nota}[thm]{Notation}
\newtheorem{cor}[thm]{Corollary}

\newtheorem{case}[thm]{Case}
\newtheorem{remark}[thm]{Remark}
\newtheorem{ex}[thm]{Example}
\numberwithin{equation}{section}

\begin{document}
\begin{abstract}
We study the random Fibonacci tree, which is an infinite binary tree with non-negative integers at each node. The root consists of the number $1$ with a single child, also the number $1$. We define the tree recursively in the following way: if $x$ is the parent of $y$, then $y$ has two children, namely $|x-y|$ and $x+y$. This tree was studied by Benoit Rittaud \cite{average} who proved that any pair of integers $a,b$ that are coprime occur as a parent-child pair infinitely often.  We extend his results by determining the probability that a random infinite walk in this tree contains exactly one pair $(1,1)$, that being at the root of the tree. Also, we give tight upper and lower bounds on the number of occurrences of any specific coprime pair $(a,b)$ at any given fixed depth in the tree.
\end{abstract}
\begin{frontmatter}
\title{On $(a,b)$ Pairs in Random Fibonacci Sequences\tnoteref{t1,t2}}
\author[add1]{Kevin G. Hare}
\ead{kghare@uwaterloo.ca}
\tnotetext[t1]{Research of K.G. Hare was supported by NSERC Grant RGPIN-2014-03154}
\address[add1]{Department of Pure Mathematics, University of Waterloo, Waterloo, Ontario, Canada N2L 3G1}
\author[add1]{J.C. Saunders\corref{cor1}}
\tnotetext[t2]{Research of J.C. Saunders was supported by NSERC Grant RGPIN-2014-03154, NSERC Grant 261908-2011, and the Queen Elizabeth II Graduate Scholarship in Science and Technology program}
\ead{j8saunde@uwaterloo.com}
\cortext[cor1]{Corresponding author}
\end{frontmatter}
\section{Introduction}
\label{sec:intro}
The Fibonacci sequence, recursively defined by $f_1=f_2=1$ and $f_n=f_{n-1}+f_{n-2}$ for all $n\geq 3$, has been generalised in several ways. In 2000, Divakar Viswanath studied random Fibonacci sequences given by $t_1=t_2=1$ and $t_n=\pm t_{n-1}\pm t_{n-2}$ for all $n\geq 3$ where the sign choosing of $\pm$ is independent for each one and each $\pm$ is replaced with $+$ or $-$ with probability $1/2$. Viswanath proved that
\begin{equation*}
\lim_{n\rightarrow\infty}\sqrt[n]{|t_n|}=1.13198824\dots
\end{equation*}
with probability $1$ \cite{viswanath}. An exact value is still unknown.

In 2006, Jeffrey McGowan and Eran Makover used the formalism of trees to give a simpler proof of Viswanath's result to evaluate the growth of the average value of the $n$th term \cite{eran}. More precisely, they proved that
\begin{equation*}
1.12095\leq\sqrt[n]{E(|t_n|)}\leq 1.23375
\end{equation*}
where $E(|t_n|)$ is the expected value of the $n$th term of the sequence.

In 2007, Rittaud used McGowan and Makover's idea of trees to construct full binary Fibonacci trees in the following way \cite{average}. The root, which is at the top, consists of a number $g_0$ with a single child $g_1$, with at least one of these two values not being $0$. Rittaud then defined the tree recursively as follows: if $x$ is the parent of $y$, then $y$ has two children, $|x-y|$ on the left branch and $x+y$ on the right branch. Rittaud denoted this tree as $\mathbf{T}_{(g_0,g_1)}$ \cite{average}, which leads to the following definition.
\begin{defn}
We say a parent child pair $(a,b)$ is at depth $n$ if there exists a walk from the root of the tree $g_0,g_1,\dots,g_{n+1},$ where $g_n=a$ and $g_{n+1}=b$.
\end{defn}
\begin{ex}
Figure \ref{fig:T} gives the top of the tree $\mathbf{T}_{(1,1)}$. There are five pairs $(1,1)$ at depth $3$.
\end{ex}
Rittaud let $m_n$ denote the mean value of the $2^{n-1}$ values that are children of pairs at depth $n$ in the tree \cite{average}. Rittaud showed that when $g_0=g_1=1$, the resulting tree will have the property that an ordered pair of natural numbers $(a,b)$ will occur on a single branch of this tree with $a$ being the parent of $b$ if and only if $\gcd(a,b)=1$. Further, he showed that any such pair occurs infinitely many times \cite{average}. In particular, this is true for the pair $(1,1)$. In 2009 Janvresse, Rittaud, and De La Rue  proved that, independent of the choices for $g_0$ and $g_1$,
\begin{equation*}
\lim_{n\rightarrow\infty}\frac{m_{n+1}}{m_n}=\alpha-1\approx 1.20556943
\end{equation*}
where $\alpha$ is the real number satisfying $\alpha^3=2\alpha^2+1$ \cite{janvresse}.

In this paper, we also consider the tree $\mathbf{T}_{(1,1)}$.
\begin{figure}
\begin{center}
\begin{tikzpicture}
[level distance=1cm,
  level 1/.style={sibling distance=1cm},
  level 2/.style={sibling distance=6.4cm},
  level 3/.style={sibling distance=3.2cm},
  level 4/.style={sibling distance=1.6cm},
  level 5/.style={sibling distance=0.8cm},
  level 6/.style={sibling distance=0.4cm}]
  \node {1}
    child {node {1}
      child {node {0}
         child{node{1}
            child{node{1}
               child{node{0}
                  child{node{1}}
                  child{node{1}}}
               child{node{2}
                  child{node{1}}
                  child{node{3}}}
            }
             child{node{1}
               child{node{0}
                  child{node{1}}
                  child{node{1}}}
               child{node{2}
                  child{node{1}}
                  child{node{3}}}
            }
         }
         child{node{1}
             child{node{1}
               child{node{0}
                  child{node{1}}
                  child{node{1}}}
               child{node{2}
                  child{node{1}}
                  child{node{3}}}
            }
            child{node{1}
               child{node{0}
                  child{node{1}}
                  child{node{1}}}
               child{node{2}
                  child{node{1}}
                  child{node{3}}}
            }
         }
      }
      child {node {2}
         child{node{1}
             child{node{1}
               child{node{0}
                  child{node{1}}
                  child{node{1}}}
               child{node{2}
                  child{node{1}}
                  child{node{3}}}
            }
             child{node{3}
               child{node{2}
                  child{node{1}}
                  child{node{5}}}
               child{node{4}
                  child{node{1}}
                  child{node{7}}}
            }
         }
         child{node{3}
             child{node{1}
               child{node{2}
                  child{node{1}}
                  child{node{3}}}
               child{node{4}
                  child{node{3}}
                  child{node{5}}}
            }
            child{node{5}
               child{node{2}
                  child{node{3}}
                  child{node{7}}}
               child{node{8}
                  child{node{3}}
                  child{node{13}}
               }
            }
         }
      }
    };
\end{tikzpicture}
\end{center}
\caption{The Top of the Tree $\mathbf{T}_{(1,1)}$}
\label{fig:T}
\end{figure}
We extend Rittaud's result on $(a,b)$ pairs occurring infinitely often by giving tight bounds on the number of such pairs at any specific depth in the tree. Let $n\in\mathbb{N}\cup\{0\}$. We observe that if $(a,b)$ is a pair at depth $3n$, then $a$ and $b$ are odd. Similarly if $(a,b)$ is a par at depth $3n+1$, then $a$ is odd and $b$ is even. Lastly if $(a,b)$ occurs at depth $3n+2$, then $a$ is even and $b$ is odd. This leads to the following definition.
\begin{defn}
\label{defn:A}
We denote by $A_{(a,b)}(n)$ the number of $(a,b)$ pairs that are found at depth $3n+m$ in the Fibonacci tree. Here we have $m=0$ if $a$ and $b$ are odd, $m=1$ if $a$ is odd and $b$ is even, and $m=2$ if $a$ is even and $b$ is odd.
\end{defn}
\begin{ex}
In Figure \ref{fig:T}, we can see that $A_{1,1}(1)=5$ and $A_{1,2}(1)=6$.
\end{ex}
In \cite {average}, Rittaud introduced the concept of a single $0$-walk. This is a walk that consists of $3n+2$ branches, starts at the root, and ends at a node with number $0$ and does not attain a $0$ elsewhere in this walk. He showed that the number of these walks is
\begin{equation*}
\frac{1}{2n+1}{3n\choose n}.
\end{equation*}
In each such walk, however, since the last node is $0$, it can be seen that the third last and second last nodes form a $(1,1)$ pair at depth $3n$ in the tree. This leads to the following definition.
\begin{defn}
\label{defn:BS}
We let $B(n)$ be the number of $(1,1)$ pairs at depth $3n$ in the tree such that the walks to these pairs do not attain a $0$. We define $S(n)$ similarly as the number of such pairs satisfying the additional constraint that the walk not attain the pair $(1,1)$ in the interior of the walk. We call these $S(n)$ pairs primitive.
\end{defn}
\begin{ex}
As can be verified in Figure \ref{fig:T}, $S(1)=5$, $B(1)=1$, $S(2)=2$, and $B(2)=3$.
\end{ex}
For $n\geq 2$, we have
\begin{equation}
S(n) \leq B(n)=\frac{1}{2n+1}{3n\choose n}\leq A_{(1,1)}(n)\label{eqn3}.
\end{equation}
Given a walk to any $(1,1)$ pair that isn't primitive, we know that this walk must go through an intermediate primitive $(1,1)$ pair. Thus we have the formula
\begin{equation}
A_{(1,1)}(n)=\sum_{i=0}^{n-1}A_{(1,1)}(i)S(n-i)\label{Aformula}.
\end{equation}
It is also worth noting that the function $B(n)$ counts the $(1,1)$ pairs whose walks begin with a right branch and which take a right branch after every intermediate $(1,1)$ pair attained. Any walk that doesn't have this property would have to attain an intermediate $0$, contradicting our definition of $B(n)$. Conversely, every walk that has this property cannot attain any node that has a $0$ for the only way to attain a $0$ is to take an immediate left branch after a pair $(1,1)$.

Using this fact, we also have the formula
\begin{equation}
B(n)=B(n-1)+\sum_{i=0}^{n-2}B(i)S(n-i)\label{Bformula}
\end{equation}
We'll be using Equations \eqref{Aformula} and \eqref{Bformula} in the rest of the paper.

We first consider how often can we avoid the pair $(1,1)$.
In Section \ref{sec:randwalk} we prove:
\begin{thm}
\label{thm:randwalk}
Consider a random walk in the tree, starting at the root $(1,1)$, with probability $p$ of choosing a right branch be $p$ and probability $1-p$ of choosing a left branch. Then the probability the walk does not contain any $(1,1)$ pair except at the root is $0$ if $p\leq 1/3$ and is
\begin{equation*}
\frac{3p-2+\sqrt{4p-3p^2}}{2}
\end{equation*}
if $p>1/3$.
\end{thm}

In the other direction, precise asymptotics for $A_{(1,1)}(n)$ are developed in Section \ref{sec:A}. Namely, we prove that
\begin{thm}
\label{thm:A}
Letting $A_{(1,1)}(n)$ be defined as above, we have
\begin{equation*}
A_{(1,1)}(n)=\frac{243\cdot 6.75^n}{4\sqrt{3\pi}n^{3/2}}-\frac{337041\cdot 6.75^n}{288\sqrt{3\pi}n^{5/2}}+O\left(\frac{6.75^n}{n^{7/2}}\right)
\end{equation*}
with the implicit constant of the error term always lying between $0$ and $\frac{1348164}{36\sqrt{3\pi}}$ for all $n\in\mathbb{N}$. That is, for sufficiently large $n$ we have 
\begin{equation*}
\left(\frac{243\cdot 6.75^n}{4\sqrt{3\pi}n^{3/2}}\right)\left(1-\frac{1387}{72n}\right)<A_{(1,1)}(n)<\left(\frac{243\cdot 6.75^n}{4\sqrt{3\pi}n^{3/2}}\right)\left(1-\frac{1387}{72n}+\frac{5548}{9n^2}\right).
\end{equation*}
\end{thm}
In Section \ref{sec:Aab}, we develop precise asymptotics for $A_{(a,b)}$ for all coprime pairs $(a,b)$. Namely, we prove that 
\begin{thm}
\label{thm:Aab}
For all coprime pairs $(a,b)$, there exists an explicitly computable positive constant $C_{(a,b)}$ and an asymptotically estimatable rational constant $D_{(a,b)}$ such that
\begin{equation*}
A_{(a,b)}(n)=\frac{C_{(a,b)}\cdot 6.75^n}{n^{3/2}}+\frac{C_{(a,b)}D_{(a,b)}6.75^n}{n^{5/2}}+O\left(\frac{6.75^n}{n^{7/2}}\right).
\end{equation*}
Here the implied constant in the error term depends upon the number of branches in the shortest walk from the root $(1,1)$ to the pair $(a,b)$.
\end{thm}

The paper is divided up as follows. The proof of Theorem \ref{thm:randwalk} is given in Section \ref{sec:randwalk}.

In Section \ref{sec:ABS} we discuss asymptotic formulas
    for $S(n)$ and $B(n)$, and a weak upper bound for $A_{(1,1)}(n)$, which will be 
    used in the proof of Theorem \ref{thm:A}.
In Section \ref{sec:coprime} we develop some preliminary results for other
    coprime pairs, which are used in both the proofs of Theorems \ref{thm:A}
    and \ref{thm:Aab}.
Sections \ref{sec:A} and \ref{sec:Aab} provide sketches of the proofs of Theorems
    \ref{thm:A} and \ref{thm:Aab} respectively.
The last section, Section \ref{sec:conc}, discusses some open questions 
     related to this research.

Many of the proofs for this paper are routine in nature, and can be found in   
    the Appendicies of the arXiv version of this paper \cite{arXiv}.

    Since this paper discusses asymptotic results involving the number $6.75$, we note the following for the rest of the paper.

\begin{note}
Whenever we write $6.75$ in this paper, we mean this exact value of $\frac{27}{4}$.
\end{note}
\begin{nota}
Suppose we have two functions $f,g:\mathbb{N}\rightarrow\mathbb{R}$. In the rest of the paper, we use the notation
\begin{equation*}
f(n)\sim g(n)
\end{equation*}
to mean that
\begin{equation*}
\lim_{n\rightarrow\infty}\frac{f(n)}{g(n)}=1.
\end{equation*}
\end{nota}
\section{Random walks in the Tree}
\label{sec:randwalk}
In this section we will consider the problem of how often we expect to find a 
    random infinite walk that never attains the pair $(1,1)$.
Here we prove Theorem \ref{thm:randwalk}.
Before this, we need some notation that is used throughout the paper.

\begin{nota}
For a coprime pair $(a,b)$, let $SW_{(a,b)}(c,d)$ denote the number of branches in the shortest walk from an $(a,b)$ pair to a $(c,d)$ pair. 
\end{nota}

In Appendix \ref{app:primitive} of the arXiv version of this paper \cite{arXiv} we give the details of the proof of
\begin{prop}
\label{prop:primitive}
Let $w_1,w_2,\dots,w_n$ be a sequence of left and right branches corresponding to a walk to a primitive $(1,1)$. Then for all $1\leq i<n$  the number of left branches in $w_1,\dots,w_i$ is strictly less than twice the number of right branches. Further, $w_1,\dots,w_n$ will contain exactly twice as many left branches as right branches.  Moreover, all walks of this form are walks to primitive $(1,1)$s.
\end{prop}

\begin{proof}[Proof of Theorem \ref{thm:randwalk}]Each walk not containing a pair $(1,1)$ except at the root must begin with a right branch. 
From there, by Proposition \ref{prop:primitive}, each desirable walk can correspond to an infinite positive integer sequence, each number denoting $SW_{(a,b)}(1,1)$ for a specific pair $(a,b)$ in the given walk. Thus we can consider the problem of having random integer sequences beginning with $2$ and either adding $2$ or subtracting $1$ to get the next number. We want to know the probability of such a sequence having all of its terms be positive.

For each $n\in\mathbb{N}\cup\{0\}$ let us denote the probability of starting a sequence with $n$ and applying the above rules and eventually traversing $0$ with $P(n)$. Thus we have the recurrence
\begin{equation}
P(n)=(1-p)P(n-1)+pP(n+2),\ \  n\neq 0\label{rec}
\end{equation}
with $P(0)=1$. Here the successor $n$ in the sequence will either be $n-1$ with probability $1-p$ or $n+2$ with probability $p$.

We can prove that there exists constants $A,B,$ and $C$ such that for all $n\in\mathbb{N}\cup\{0\}$, we have
\begin{equation}
P(n)=A+Br_1^n+Cr_2^n
\end{equation}
where
\begin{equation*}
r_1=\frac{-1+\sqrt{4/p-3}}{2}\text{ and }r_2=\frac{-1-\sqrt{4/p-3}}{2}.
\end{equation*}
From \eqref{rec}, we obtain for $n\geq 3$ that
\begin{equation*}
P(n)=\frac{P(n-2)}{p}-\frac{P(n-3)(1-p)}{p}.
\end{equation*}

We recall that $P(2)$ is the probability of starting at $2$ and randomly adding $2$ or subtracting $1$ and eventually traversing $0$. Thus $1-P(2)$ is the probability of never traversing $0$. This is equal to the probability of never traversing a pair $(1,1)$ in the Fibonacci tree after taking the first branch to be a right branch. Since it could be the first branch is a left branch (from which it is unavoidable to attain another $(1,1)$ pair), we therefore have that the probability of a random walk not traversing a pair $(1,1)$ except at the root is $p (1-P(2))$.
In Appendix \ref{app:randwalk} of the arXiv version of this paper \cite{arXiv} we show that 
\begin{itemize}
\item If $p \leq 1/3$ then $P(n) = 1$
\item If $p > 1/3$ then $P(n) = r_1^n$.
\end{itemize}
From this it follows that if $p \leq 1/3$ then the probability that a walk does not contain any $(1,1)$ is $0$. 
If instead $p > 1/3$ the probability is 
\begin{align*}
p(1-P(2))&=p(1-r_1^2)\\
& =\frac{3p-2+p\sqrt{4/p-3}}{2}.
\end{align*}
\end{proof}

\section{Preliminary results for $A_{(1,1)}(n), B(n)$ and $S(n)$}
\label{sec:ABS}

Recall that $B(n)$ counts the number of $(1,1)$ pairs at depth $3n$ in the tree such that the walk does not attain a $0$, whereas $S(n)$ is defined similarly
    except the walk does not attain an intermediate pair $(1,1)$.

In Appendix \ref{app:ABS} of the arXiv version of this paper \cite{arXiv} we prove that
\begin{equation*}
S(n)=\frac{6.75^n}{3\sqrt{3\pi}n^{3/2}}\left(1+\frac{17}{72n}+O\left(\frac{1}{n^2}\right)\right)
\end{equation*}
and
\begin{equation*}
B(n)=\frac{\sqrt{3}\cdot 6.75^n}{4\sqrt{\pi}n^{3/2}}\left(1-\frac{43}{72n}+O\left(\frac{1}{n^2}\right)\right).
\end{equation*}
In fact, a tighter version of these results is in Corollary \ref{cor:3.1} in this appendix.

Recall that $A_{(1,1)}(n)$ is the number of pairs $(1,1)$ pairs at depth $3n$ in the tree where there 
    are no restrictions on the walk.
We further show that 

\begin{prop}
\label{prop:Aratio}
We have
\begin{equation*}
\lim_{n\rightarrow\infty}\frac{A_{(1,1)}(n+1)}{A_{(1,1)}(n)}=6.75.
\end{equation*}
\end{prop}

These results are both used in the proof of Theorem \ref{thm:A}.

\section{Preliminary Results Concerning Other Coprime Pairs}
\label{sec:coprime}

In this section, we turn our attention to the behaviour of other coprime pairs other than $(1,1)$ and establish a number of useful preliminary results concerning them.

In Appendix \ref{app:coprime} of the arXiv version of this paper \cite{arXiv} we give the details of the proof of
\begin{prop}
\label{prop:coprimepair3}
Take a coprime pair $(a,b)$ that is not the $(1,1)$ pair and suppose that $SW_{(1,1)}(a,b)=k$. Then we have for all $n\geq\lfloor\frac{k}{3}\rfloor$
\begin{equation*}
A_{(a,b)}(n)=\sum_{i=\lfloor\frac{k-1}{3}\rfloor}^nA_{(|a-b|,a)}(i)B(n-i)
\end{equation*}
if either $a$ is even or $b$ is even, and
\begin{equation*}
A_{(a,b)}(n)=\sum_{i=\lfloor\frac{k-1}{3}\rfloor}^{n-1}A_{(|a-b|,a)}(i)B(n-1-i)
\end{equation*}
if $a$ and $b$ are both odd.
\end{prop}

Some important corollaries of Proposition \ref{prop:coprimepair3} are:
\begin{cor}
\label{cor:5.2}
For all $n\in\mathbb{N}\cup\{0\}$, we have
\begin{equation*}
A_{(1,2)}(n)=\frac{A_{(1,1)}(n+1)-B(n+1)}{4}.
\end{equation*}
\end{cor}
\begin{proof}
By Proposition \ref{prop:coprimepair3}, we have for all $n\in\mathbb{N}\cup\{0\}$
\begin{equation}
A_{(1,2)}(n)=\sum_{i=0}^{n}A_{(1,1)}(i)B(n-i)\label{A12formula}.
\end{equation}
All walks down to a $(1,1)$ pair at depth $3n$ in the tree must satisfy exactly one of the following two conditions. Either for all other $(1,1)$ pairs it attains it takes a right branch immediately afterwards, or there exists a first $(1,1)$ pair where the walk takes a left branch immediately afterwards, consequently ending up immediately at a choice of $4$ $(1,1)$ pairs. Thus for  all $n\in\mathbb{N}\cup\{0\}$ we have
\begin{equation*}
A_{(1,1)}(n)=B(n)+\sum_{i=0}^{n-1}B(i)\cdot 4\cdot A_{(1,1)}(n-1-i).
\end{equation*}
Relabeling the index in the summation gives
\begin{equation}
A_{(1,1)}(n)=B(n)+4\sum_{i=0}^{n-1}A_{(1,1)}(i)B(n-1-i)\label{ABformula}.
\end{equation}
Substituting in \eqref{A12formula} we have for all $n\in\mathbb{N}$
\begin{equation*}
A_{(1,1)}(n)=B(n)+4\cdot A_{(1,2)}(n-1).
\end{equation*}
Thus we have our result.
\end{proof}
\begin{cor}
\label{cor:5.3}
Take two pairs of coprime positive integers $(a,b)$ and $(c,d)$ and suppose that $SW_{(1,1)}(a,b)=SW_{(1,1)}(c,d)$. Then we have
\begin{equation*}
A_{(a,b)}(n)=A_{(c,d)}(n)
\end{equation*}
for all $n\in\mathbb{N}\cup\{0\}$.
\end{cor}
\begin{proof}
We can prove this by induction on the number of branches in the shortest walks, using the result of Proposition \ref{prop:coprimepair3}.
\end{proof}
\begin{cor}
\label{cor:5.4}
For all $n\in\mathbb{N}\cup\{0\}$, we have
\begin{equation*}
A_{(2,1)}(n)=A_{(2,3)}(n)=A_{(1,1)}(n+1)-4\cdot A_{(1,1)}(n).
\end{equation*}
\end{cor}
\begin{proof}
By Corollary \ref{cor:5.3} it suffices to prove that
\begin{equation*}
A_{(2,1)}(n)=A_{(1,1)}(n+1)-4\cdot A_{(1,1)}(n).
\end{equation*}
since the pairs $SW_{(1,1)}(2,1)=SW_{(1,1)}(2,3)=2$. First, consider all $(2,1)$ pairs at depth $3n+2$ in the tree. If we take an immediate left branch we encounter $(1,1)$ pairs at depth $3n+3$ in the tree. Now consider all $(1,1)$ pairs at depth $3n+3$ in the tree. The walks to these $(1,1)$ pairs must either have the element $0$ or the element $2$ immediately before the final $(1,1)$ pair. There are $4\cdot A_{(1,1)}(n)$ $(1,1)$ pairs of the former type since following backwards along the walk will give us a $(1,1)$ pair at depth $3n$ in the tree and each of these $(1,1)$ pairs at depth $3n$ produces four paths of the sequence $1,1,0,1,1$.  Therefore the number of $(1,1)$ pairs with a walk that has the element $2$ immediately before the $(1,1)$ pair is $A_{(1,1)}(n+1)-4\cdot A_{(1,1)}(n)$. Since the second and third last elements of these walks form $(2,1)$ pairs we have, by our observation that all $(2,1)$ pairs have a $(1,1)$ immediately beneath them, our result.
\end{proof}

The proof of the following results can be found in Appendix \ref{app:coprime} of the arXiv version of this paper \cite{arXiv}.
\begin{lemma}
\label{lem:5.7}
Take a coprime pair $(a,b)$ that is not the $(1,1)$ pair and suppose that $SW_{(1,1)}(a,b)=k\geq 3$. Suppose the last five numbers in the corresponding sequence of the shortest walk, including the last two numbers $a$ and $b$, are
\begin{equation*}
a_0,a_1,a_2,a,b.
\end{equation*}
For all $n\geq\lfloor\frac{k}{3}\rfloor$, we have
\begin{equation*}
A_{(a,b)}(n)=A_{(a_1,a_2)}(n)-A_{(a_0,a_1)}(n)
\end{equation*}
if $a$ is odd, and we have
\begin{equation*}
A_{(a,b)}(n)=A_{(a_1,a_2)}(n+1)-A_{(a_0,a_1)}(n).
\end{equation*}
if $a$ is even (and hence $b$ is odd).
\end{lemma}

\begin{lemma}
\label{lem:6.1}
For all $n\in\mathbb{N}\cup\{0\}$, we have
\begin{equation*}
A_{(1,1)}(n+2)-16\cdot A_{(1,1)}(n+1)+64\cdot A_{(1,1)}(n)=B(n+2)+4\cdot S(n+2).
\end{equation*}
\end{lemma}

\begin{prop}
\label{prop:A}
We have
\begin{equation*}
A_{(1,1)}(n)=\frac{243\cdot 6.75^n}{4\sqrt{3\pi}n^{3/2}}(1+o(1)).
\end{equation*}
\end{prop}

These results are important for the proofs of Theorem \ref{thm:A} and \ref{thm:Aab}.

\section{Proof of Theorem \ref{thm:A}}
\label{sec:A}

To help prove tight bounds for $A_{(1,1)}(n)$, we use an auxilliary function $D(n)$, defined below, along with
    asymptotic information about $B(n)$ and $S(n)$.
\begin{defn}
\label{defn:D}
Define $D:\mathbb{N}\cup\{0\}\rightarrow\mathbb{Z}$ as $D(0)=-3$ and
\begin{equation*}
D(n+1)=8\cdot D(n)+B(n+2)+4\cdot S(n+2).
\end{equation*}
for all $n\in\mathbb{N}\cup\{0\}$. It can be verified with the help of Lemma \ref{lem:6.1} that, for all $n\in\mathbb{N}\cup\{0\}$, we have
\begin{equation}
D(n)=A_{(1,1)}(n+1)-8\cdot A_{(1,1)}(n)\label{eqn11}.
\end{equation}
\end{defn}
\begin{lemma}
\label{lem:6.2}
We have
\begin{equation*}
D(n)\sim\frac{-405\sqrt{3}\cdot 6.75^n}{16\sqrt{\pi}n^{3/2}}.
\end{equation*}
\end{lemma}
\begin{proof}
The lemma can be verified with \eqref{eqn11} and Proposition \ref{prop:A}. 
\end{proof}

The proof of the following proposition, used in the proof of Theorem \ref{thm:A} can be found in Appendix \ref{app:D} of the arXiv version of this paper \cite{arXiv}.
\begin{prop}
\label{prop:Dinequality}
For all $n\geq 100$, we have
\begin{align*}
\left(-\frac{405\sqrt{3}\cdot 6.75^n}{16\sqrt{\pi}n^{3/2}}\right)\left(1-\frac{4019}{360n}+\frac{207}{n^2}\right)& <D(n) \\
 & <\left(-\frac{405\sqrt{3}\cdot 6.75^n}{16\sqrt{\pi}n^{3/2}}\right)\left(1-\frac{4019}{360n}\right).
\end{align*}
\end{prop}

\begin{proof}[Proof of Theorem \ref{thm:A}]
The proof of the desired inequalities follows the same procedure as in the proof of Proposition \ref{prop:Dinequality}. We prove by contradiction in the following way. We first assume that the desired upper bound does not hold for some value of $n\geq 100$. Using \eqref{eqn11} and the lower bound for $D(n)$ in Proposition \ref{prop:Dinequality}, we derive a lower bound for $A_{(1,1)}(n+1)$. Again, we see that $A_{(1,1)}(n+1)$ does not satisfy the desired upper bound given in the Theorem so that we can repeat the argument to get a lower bound for $A_{(1,1)}(n+2)$ and so on. As $k\rightarrow\infty$, we see that the error term for $A_{(1,1)}(n+k)$ grows too big, overwhelming the main term, contradicting Proposition \ref{prop:A}. The proof for the lower bound works the same way, using \eqref{eqn11} and the upper bound for $D(n)$.
\end{proof}
\begin{remark}
Theorem \ref{thm:A} provides very good estimates for $A_{(1,1)}(n)$ for all $n\in\mathbb{N}$. Define $(C_n)_{n\in\mathbb{N}}$ by
\begin{equation*}
A_{(1,1)}(n)=\frac{C_n\cdot 6.75^n}{n^{3/2}}.
\end{equation*}
Then Theorem \ref{thm:A} gives us
\begin{equation*}
C\left(1-\frac{1387}{72n}\right)<C_n<C\left(1-\frac{1387}{72n}+\frac{5548}{9n^2}\right)
\end{equation*}
where
\begin{equation*}
C=\frac{243}{4\sqrt{3\pi}}=19.78840173\dots
\end{equation*}
For example, we have
\begin{equation*}
19.7502<C_{10000}<19.7505.
\end{equation*}
\end{remark}

\section{Proof of Theorem \ref{thm:Aab}}
\label{sec:Aab}

Finally, we establish our asymptotic results for other coprime pairs $A_{(a,b)}(n)$ for all coprime ordered pairs $(a,b)$. First, from Theorem \ref{thm:A} and using results from Section \ref{sec:ABS}, we can derive the asymptotic formulas for the pairs $(1,2)$, $(2,1)$, and $(2,3)$.

We are now ready to prove our main result concerning the asymptotic formulas for all coprime pairs $(a,b)$.
\begin{proof}[Proof of Theorem \ref{thm:Aab}]
First, we claim that the constants $C_{(a,b)}$ in the Theorem have the form
\begin{equation*}
C_{(a,b)}=\frac{243t_k}{4\sqrt{3\pi}}
\end{equation*}
if $SW_{(1,1)}(a,b)=k$ where for all $k\in\mathbb{N}\cup{0}$ we have
\begin{equation*}
t_{3k}=\left(\frac{1}{2}\right)^k\left(1+\frac{k}{3}\right),
\end{equation*}
\begin{equation*}
t_{3k+1}=\left(\frac{1}{2}\right)^k\left(\frac{5}{3}+\frac{k}{2}\right),
\end{equation*}
and
\begin{equation*}
t_{3k+2}=\left(\frac{1}{2}\right)^k\left(\frac{11}{4}+\frac{3k}{4}\right).
\end{equation*}
Note that $t_0=1$, $t_1=\frac{5}{3}$, and $t_2=\frac{11}{4}$. One can verify that, for all $k\geq 3$, $t_k$ satisfies the recurrence
\begin{equation*}
t_k=t_{k-2}-t_{k-3}
\end{equation*}
if $3\nmid k+1$ and
\begin{equation*}
t_k=6.75\cdot t_{k-2}-t_{k-3}.
\end{equation*}
if $3|k+1$. Also, for the constants $D_{(a,b)}$, we define the sequence $(s_k)_k\in\mathbb{N}\cup\{0\}$ and $s_k:=D_{(a,b)}$ if $SW_{(1,1)}(a,b)=k$. By Corollary \ref{cor:5.3}, this sequence is well-defined. We further claim that for all $k\geq 3$, we have
\begin{equation*}
s_k=\frac{t_{k-2}s_{k-2}}{t_k}-\frac{s_{k-3}t_{k-3}}{t_k}
\end{equation*}
if $3\nmid k+1$ and
\begin{equation*}
s_k=\frac{t_{k-2}(2s_{k-2}-3)}{t_k}-\frac{s_{k-3}t_{k-3}}{t_k}
\end{equation*}
if $3|k+1$.
We prove both of these claims by induction in Appendix \ref{app:Aab} of the arXiv version of this paper \cite{arXiv}.

Let $a_k=s_kt_k$. We can verify that
\begin{equation*}
\lim_{k\rightarrow\infty}t_k=0
\end{equation*}
so that for $k$ sufficiently large we have by our recursive formulas for $s_k$ that
\begin{equation*}
a_k\approx a_{k-2}-a_{k-3}.
\end{equation*}
Solving this recurrence relation in much the same way we solved the recurrence relation in Theorem \ref{thm:randwalk} gives the asymptotic
\begin{equation*}
a_k\approx C\cdot(-1.3247\dots)^k
\end{equation*}
for some constant $C$ where $-1.3247\dots$ is the only real root of
\begin{equation*}
x^3-x-1.
\end{equation*}
Thus we obtain
\begin{equation*}
s_k\approx\frac{C'\cdot(-2.6494\dots)^k}{k}
\end{equation*}
where $C'$ depends on $k\mod 3$.
\end{proof}

\section{Further Questions}
\label{sec:conc}

On counting the number of $(a,b)$ pairs in the Fibonacci Tree, there are still alot of questions that have been left unanswered. Some of these are as follows. Can we get even tighter bounds for $A_{(1,1)}(n)$? Theorem \ref{thm:A} above was essentially derived from Robbins' bounds for factorials. Since Robbins, however, there have been numerous improvements on bounds for factorials that will probabily help us derive even better bounds for $A_{(1,1)}(n)$. For example, Knopp \cite{knopp} shows that there exists constants $A,B,C,D,\dots$ such that the sequence
\begin{equation*}
r_n:=\ln\left(\frac{n!e^n}{\sqrt{2\pi}n^{n+1/2}}\right)
\end{equation*}
is bounded above and below by the partial sums of
\begin{equation*}
\frac{A}{n}-\frac{B}{n^3}+\frac{C}{n^5}-\frac{D}{n^7}+\dots
\end{equation*}
 and Impens \cite{impens} shows how to compute those constants recursively. We may be able to use these results to prove that there exists positive constants $A,B,C,D,\dots$ such that
\begin{equation*}
\frac{A_{(1,1)}(n)\cdot 4\sqrt{3\pi}n^{3/2}}{243\cdot 6.75^n}
\end{equation*}
can be approximated by
\begin{equation*}
A+\frac{B}{n}+\frac{C}{n^2}+\frac{D}{n^3}+\dots
\end{equation*}
In this paper, we showed that $A=1$, $B=\frac{-1387}{72}$ and that, if $C$ exists, then $0\leq C\leq\frac{5548}{9}$. We may be able to use the same procedure as in this paper to derive more terms of this series. Analogous questions remain open for $A_{(a,b)}$ for all coprime ordered pairs $(a,b)$. As another direction, what is the probability of a walk in the Fibonacci tree containing exactly $k$ occurances of $(1,1)$ where $k\in\mathbb{N}$?

We can also look at variations of the Fibonacci Tree. For example, in taking a left branch from the ordered pair $(x,y)$ do a subtraction $x-y$ instead of taking the mere difference $|x-y|$ or more generally for some $k\in\mathbb{N}$, take $k$ children all of them being $x+\delta y$ where $\delta$ is a different $k$th root of unity for each one.
\section{Acknowledgements}
The authors would like to thank Dr. Yu-Ru Liu for her support and suggestions with this paper.


\newpage
\section*{Appendix}

\setcounter{section}{0}
\renewcommand{\thefigure}{\Alph{figure}}
\renewcommand{\thesection}{\Alph{section}}

\section{Proof of Proposition \ref{prop:primitive}}
\label{app:primitive}

In this section, we prove Proposition \ref{prop:primitive}.
We first need some preliminary lemmas.
\begin{lemma}
\label{lem:2.1}
Let $a, b, a_1, \dots, a_n, c, d$ be a walk from $(a,b)$ to $(c,d)$. Then $d, c, a_n, \dots, a_2, a_1, b, a$ is a walk from $(d,c)$ to $(b,a)$.
\end{lemma}
\begin{proof}
We only have to show that if $a_1,a_2,a_3$ occur in the given walk, then $a_3,a_2,a_1$ can consecutively occur in a walk in that order. We have either $a_3=a_1+a_2$ or $a_3=|a_2-a_1|$. In the first case, we have $a_1=|a_2-a_3|$ giving us our result. In the second case, we either have $a_3=a_2-a_1$, giving us $a_1=a_2-a_3$, or $a_3=a_1-a_2$, giving us $a_1=a_2+a_3$.
\end{proof}
\begin{lemma}
\label{lem:2.2}
The shortest walk from a non-$(1,1)$ pair to a $(1,1)$ pair is characterised as a series of left branches with no right branches.
\end{lemma}
\begin{proof}
By Lemma \ref{lem:2.1}, we obtain the shortest walk by traversing backwards along the shortest walk from $(1,1)$ to the given pair non-$(1,1)$ pair $(a,b)$. In \cite[Corollary 5.1]{average}, Rittaud observes that the latter walk has the property that for any pair $(c,d)$ occurring in the walk, the parent of $c$ is $|c-d|$. Thus the shortest walk from $(a,b)$ to $(1,1)$ must have the property that for any pair $(c,d)$ occurring in the walk, the child of $d$ is $|c-d|$, a choice of a left branch. Thus the shortest walk must contain no right branches.
\end{proof}
\begin{note}
We say the walk from a pair $(a,b)$ to another pair $(c,d)$ consists of $n$ branches if there are exactly $n$ branches between the node at $b$ of the first pair to the node at $d$ of the second pair.
\end{note}
\begin{lemma}
\label{lem:2.3}
Starting from a non-$(1,1)$ pair $(a,b)$ in the tree, suppose $SW_{(a,b)}(1,1)=n$. Then $SW_{(b,|b-a|)}(1,1)=n-1$ and $SW_{(b,a+b)}(1,1)=n+2$.
\end{lemma}
\begin{proof}
Follows from Lemma \ref{lem:2.2}.
\end{proof}

\begin{figure}
\begin{center}
\begin{tikzpicture}
[level distance=1cm,
level 1/.style={sibling distance=1cm},
  level 2/.style={sibling distance=6.4cm},
  level 3/.style={sibling distance=3.2cm},
  level 4/.style={sibling distance=1.6cm},
  level 5/.style={sibling distance=0.8cm},
  level 6/.style={sibling distance=0.4cm},
  level7/.style={sibling distance=0.4cm}]
  \node {1}
    child {node {1}
      child {node {0}
      }
      child {node {2}
         child{node{1}
             child{node{1}
                           }
             child{node{3}
               child{node{2}
                  child{node{1}
                     child{node{1}}
                     child{node{3}}
                     }
                  child{node{5}}}
               child{node{4}
                  }
            }
         }
         child{node{3}
         }
      }
    };
\end{tikzpicture}
\end{center}
\caption{A walk from the root $(1,1)$ to a primitive $(1,1)$ pair at depth $6$ in the tree}
\end{figure}
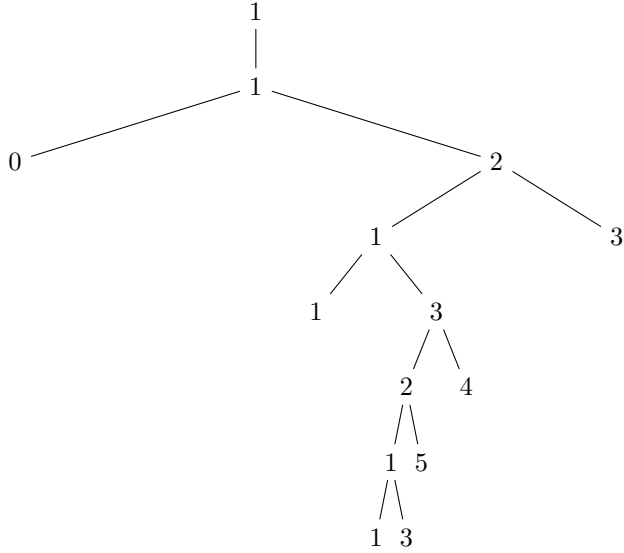

\begin{proof}[Proof of Proposition \ref{prop:primitive}]
The fact that the first branch has to be a right branch follows from the observation that a left branch will just lead to all $(1,1)$ pairs at depth $3$ in the tree. The first right branch consists of the pair $(1,2)$. From here the shortest walk to a $(1,1)$ pair consists of two left branches. Suppose we have a walk from this $(1,2)$ to a primitive $(1,1)$. Let the nodes in this walk be $1,2,a_1,a_2,\dots,a_k,1,1$. Consider the sequence 
    \[ SW_{(1,2)}(1,1), SW_{(2,a_1)}(1,1), SW_{(a_1,a_2)}(1,1), \dots, SW_{(a_k,1)}(1,1), SW_{(1,1)}(1,1).\] 
This is a sequence of integers starting with $2$ (since $SW_{(1,2)}(1,1)=2$). Each successive element in the sequence is obtained by adding $2$ to the previous element (representing going down a right branch) or subtracting $1$ from the previous element (representing going down a left branch) by Lemma \ref{lem:2.3}. Finally, all integers in the sequence will be positive, with the exception of the last being $0$ since $SW_{(1,1)}(1,1)=0$.

One property of such a sequence is that if $r$ is the number of times you add $2$, then $2r+2$ must be the number of times you subtract $1$. Moreover anywhere in the sequence except at the last element if $s$ is the number of times you added $2$ up to that point, then you cannot have subtracted $1$ more than $2s+1$ times. Moreover, it is seen that if we have a finite integer sequence starting with $2$ with the above rules in play, then all the elements in the sequence will be positive except for the last one, which will be a $0$.

Thus the walks to all the primitive $(1,1)$s in the tree that have a length of more than $3$ branches can be characterised as in the proposition.
\end{proof}

\section{Details of proof of Theorem \ref{thm:randwalk}}

\label{app:randwalk}

Recall there exists constants $A,B,$ and $C$ such that for all $n\in\mathbb{N}\cup\{0\}$, we have
\begin{equation}
P(n)=A+Br_1^n+Cr_2^n\label{pform}
\end{equation}
where
\begin{equation*}
r_1=\frac{-1+\sqrt{4/p-3}}{2}\text{ and }r_2=\frac{-1-\sqrt{4/p-3}}{2}.
\end{equation*}
From \eqref{rec}, we obtain for $n\geq 3$ that
\begin{equation*}
P(n)=\frac{P(n-2)}{p}-\frac{P(n-3)(1-p)}{p}.
\end{equation*}

\begin{case}{$p<1/3$}

We can work out that $r_1>1$ and $r_2<-2$. Therefore, if $B\neq 0$ or $C\neq 0$, then by \eqref{pform} we have
\begin{equation*}
\limsup_{n\rightarrow\infty}P(n)=\infty,
\end{equation*}
a contradiction since $0\leq P(n)\leq 1$ for all $n\in\mathbb{N}$. Therefore $B=C=0$ and since $P(0)=1$, we have
\begin{equation*}
P(n)=1
\end{equation*}
for all $n\in\mathbb{N}\cup\{0\}$. Therefore the probability of a random walk not traversing a pair $(1,1)$ except at the root is
\begin{equation*}
p(1-P(2))=p(1-1)=0.
\end{equation*}
\end{case}
\begin{case}{$p=1/3$}

We can work out that $r_1=1$ and $r_2=-2$ so that from \eqref{rec} we get
\begin{equation*}
P(n)=A+B+C(-2)^n.
\end{equation*}
If $C\neq 0$, then we have
\begin{equation*}
\limsup_{n\rightarrow\infty}P(n)=\infty,
\end{equation*}
a contradiction since $0\leq P(n)\leq 1$. Therefore $C=0$ and since $P(0)=1$, we have
\begin{equation*}
P(n)=1
\end{equation*}
for all $n\in\mathbb{N}\cup\{0\}$ and we proceed as in Case $1$.
\end{case}
\begin{case}{$1/3<p<1$}

We have
\begin{equation*}
r_2=\frac{-1-\sqrt{4/p-3}}{2}<\frac{-1-\sqrt{4-3}}{2}=-1.
\end{equation*}
Therefore, if $C\neq 0$, then by \eqref{pform}, we have
\begin{equation*}
\limsup_{n\rightarrow\infty}P(n)=\infty,
\end{equation*}
a contradiction since $0\leq P(n)\leq 1$. Therefore $C=0$ and
\begin{equation*}
P(n)=A+Br_1^n.
\end{equation*}
We will show that $A=0$ by showing $\lim_{n\rightarrow\infty}P(n)=0$. Suppose we start with $n\in\mathbb{N}$ and eventually attain $0$. Then the number of times we added $2$ is $r$ and the number of times we subtracted $1$ is  $2r+n$ for some $r\in\mathbb{N}$. Thus we have
\begin{equation*}
P(n)\leq\sum_{r=0}^{\infty}{3r+n\choose r}p^r(1-p)^{2r+n}.
\end{equation*}
In \cite{identity}, we have the combinatorial identity
\begin{equation*}
\sum_{k=0}^{n}{tk+r\choose k}{tn-tk+s\choose n-k}\frac{r}{tk+r}={tn+r+s\choose n},
\end{equation*}
which is valid for all $n\in\mathbb{N}\cup\{0\}$ and all real $r$, $s$, and $t$. Substituting in $t=3,r=1$ gives
 \begin{equation*}
\sum_{k=0}^{n}{3k+1\choose k}{3n-3k+s\choose n-k}\frac{1}{3k+1}={3n+s+1\choose n}.
\end{equation*}
Using this identity, we can prove by induction on $n\in\mathbb{N}$ that
\begin{equation*}
\sum_{r=0}^{\infty}{3r+n\choose r}p^r(1-p)^{2r+n}=(1-p)^n\left(\sum_{r=0}^{\infty}{3r\choose r}p^r(1-p)^{2r}\right)\left(\sum_{r=0}^{\infty}{3r+1\choose r}\frac{p^r(1-p)^{2r}}{3r+1}\right)^n.
\end{equation*}
Since $1/3<p\leq 1$, we can further deduce that
\begin{equation*}
\lim_{r\rightarrow\infty}\frac{{3r\choose r}}{{3(r+1)\choose r+1}}=\frac{4}{27}>p(1-p)^2
\end{equation*}
so that
\begin{equation*}
\sum_{r=0}^{\infty}{3r\choose r}p^r(1-p)^{2r}<\infty.
\end{equation*}
Moreover, entering the command ``with(SumTools)" and then the command ``$DefiniteSummation\left({3r+1\choose r}\frac{p^r(1-p)^{2r}}{3r+1},r=0..\infty\right), \text{ assuming } 0<r\leq 1$" into Maple gives
\begin{equation*}
\sum_{r=0}^{\infty}{3r+1\choose r}\frac{p^r(1-p)^{2r}}{3r+1}=\frac{2\sqrt{3}\sin\left(\frac{1}{3}\cdot \arcsin\left(\frac{3\sqrt{3}\cdot(1-p)\sqrt{p}}{2}\right)\right)}{3(1-p)\sqrt{p}}.
\end{equation*}
Thus
\begin{equation*}
\sum_{r=0}^{\infty}{3r+1\choose r}\frac{p^r(1-p)^{2r+1}}{3r+1}=\frac{2\sin\left(\frac{1}{3}\cdot \arcsin\left(\frac{3\sqrt{3}\cdot(1-p)\sqrt{p}}{2}\right)\right)}{\sqrt{3p}}.
\end{equation*}
For $1/3<p\leq 1$, we have
\begin{align*}
\frac{2\sin\left(\frac{1}{3}\cdot \arcsin\left(\frac{3\sqrt{3}\cdot(1-p)\sqrt{p}}{2}\right)\right)}{\sqrt{3p}}&\leq\frac{2\sin\left(\frac{\pi}{6}\right)}{\sqrt{3p}}\\
&=\frac{1}{\sqrt{3p}}\\
&<1.
\end{align*}
Thus we deduce
\begin{equation}
\lim_{n\rightarrow\infty}P(n)=0.\label{lim1}
\end{equation}
Thus, using \eqref{lim1}, we have
\begin{equation*}
0=\lim_{n\rightarrow\infty}P(n)=\lim_{n\rightarrow\infty}A+Br_1^n=\lim_{n\rightarrow\infty}A+Br_1^n=A,
\end{equation*}
giving us $A=0$. Thus
\begin{equation*}
P(n)=Br_1^n.
\end{equation*}
Since $P(0)=1$, we thus have $B=1$ so
\begin{equation*}
P(n)=r_1^n.
\end{equation*}
\end{case}

\section{Preliminary results concerning $A_{(1,1)}(n), B(n)$ and $S(n)$}
\label{app:ABS}
As we are dealing exclusively with $A_{(1,1)}(n)$ in this section, we will 
   use $A(n) := A_{(1,1)}(n)$ for convenience.

Recall that $B(n)$ counts the number of $(1,1)$ pairs at depth $3n$ in the tree such that the walk does not attain a $0$, whereas $S(n)$ is defined similarly
    except the walk does not attain an intermediate pair $(1,1)$.

Here we prove that
\begin{equation*}
S(n)=\frac{6.75^n}{3\sqrt{3\pi}n^{3/2}}\left(1+\frac{17}{72n}+O\left(\frac{1}{n^2}\right)\right)
\end{equation*}
and
\begin{equation*}
B(n)=\frac{\sqrt{3}\cdot 6.75^n}{4\sqrt{\pi}n^{3/2}}\left(1-\frac{43}{72n}+O\left(\frac{1}{n^2}\right)\right).
\end{equation*}
\begin{prop}
\label{prop:S}
We have $S(1)=5$ and
\begin{equation*}
S(n)=\frac{2}{3n-1}{3n-1\choose n-1}
\end{equation*}
for $n\geq 2$.
\end{prop}
\begin{proof}
We see that $S(1)=5$, and $S(2) = \frac{2}{5} {5 \choose 1} = 2$. At depth $3n$ in the tree, where $n\geq 2$, we know that if we attain a $(1,1)$ pair, then we must have taken twice as many left branches as right branches. Also, if our first branch is a left branch we will attain a $(1,1)$ pair at depth $3$ in the tree. Therefore all primitive $(1,1)$ pairs at depth $3n$ in the tree, $n\geq 2$, must occur on walks where the initial branch is a right branch. After this initial right branch, the rest of the walk must consist of $n-1$ right branches and $2n$ left branches to reach a primitive $(1,1)$ pair at depth $3n$ in the tree for $n\geq 2$. Therefore for $n\geq 2$, we have $S(n)\leq{3n-1\choose n-1}$. This upperbound, however, will over-count the number of primitive $(1,1)$s since it also counts walks where the walk to an intermediate pair might have twice as many left branches as right branches. There are ${3n-3k\choose n-k}S(k)$ such walks where the first intermediate pair with this property occurs at depth $3k$ in the tree if $k\geq 2$. If $k=1$, there are ${3n-3\choose n-1}$ such walks. For any intermediate pair we want the number of left branches to be strictly less than twice the number of right branches and so we subtract these terms to get the recurrence:
\begin{equation}
S(n)={3n-1\choose n-1}-{3n-3\choose n-1}-\sum_{k=2}^{n-1}{3n-3k\choose n-k}S(k)\label{eqn1}
\end{equation}
Assuming by induction that $S(1) = 5$ and 
\begin{equation*}
S(k)=\frac{2}{3k-1}{3k-1\choose k-1}
\end{equation*}
for $2\leq k<n$, one can check via Maple that equation \eqref{eqn1} is satisfied when $S(n) = \frac{2}{3 n -1}{3n -1\choose n-1}$ in the following way. Enter the command ``with(SumTools):" and then the command ``$DefiniteSummation\left({3n-3k\choose n-k}\cdot\frac{2}{3k-1}{3k-1\choose k-1}\right), k=1..n$", which gives the identity
\begin{equation*}
\sum_{k=1}^{n}{3n-3k\choose n-k}\frac{2}{3k-1}{3k-1\choose k-1}=\frac{1}{3}{3n\choose n}={3n-1\choose n-1}.
\end{equation*}
We can deduce \eqref{eqn1} from this induction step.
\end{proof}
\begin{prop}
\label{prop:binomialinequality}
For all $n\in\mathbb{N}$, we have
\begin{equation*}
\frac{\sqrt{3}\cdot 6.75^n}{2\sqrt{\pi n}}\left(1-\frac{7}{72n}\right)<{3n\choose n}<\frac{\sqrt{3}\cdot 6.75^n}{2\sqrt{\pi n}}\left(1-\frac{7}{72n}+\frac{1}{50n^2}\right).
\end{equation*}
\end{prop}
\begin{proof}
Robbins shows in \cite{robbins} that, for all $n\in\mathbb{N}$, we have
\begin{equation}
\sqrt{2\pi}n^{n+1/2}e^{-n}\cdot e^{1/(12n+1)}<n!<\sqrt{2\pi}n^{n+1/2}e^{-n}\cdot e^{1/(12n)}.\label{eqn10}
\end{equation}
Note the following:
\begin{equation*}
\frac{-7}{72n}<\frac{1}{36n}-\frac{1}{24n+1}-\frac{1}{12n+1}<\frac{-7}{72n}+\frac{5}{576n^2}.
\end{equation*}
Thus we have
\begin{equation*}
{3n\choose n}=\frac{(3n)!}{(2n)!\cdot n!}<\frac{\sqrt{3}\cdot 6.75^n\cdot e^{\frac{-7}{72n}+\frac{5}{576n^2}}}{2\sqrt{\pi n}}
\end{equation*}
For $-1<x<1$, we have
\begin{equation*}
e^x=1+x+\frac{x^2}{2}+\dots+\frac{x^n}{n!}+\dots
\end{equation*}
Letting $x=\frac{-7}{72n}+\frac{5}{576n^2}$, we have for $n \geq 1$ that $-1<x<0$ and hence
\begin{equation*}
e^x<1+x+\frac{x^2}{2},
\end{equation*}
as it is an alternating series. We have
\begin{align*}
{3n\choose n}&<\frac{\sqrt{3}\cdot 6.75^n}{2\sqrt{\pi n}}\left(1-\frac{7}{72n}+\frac{5}{576n^2}+\frac{\left(-\frac{7}{72n}+\frac{5}{576n^2}\right)^2}{2}\right)\\
&<\frac{\sqrt{3}\cdot 6.75^n}{2\sqrt{\pi n}}\left(1-\frac{7}{72n}+\frac{1}{50n^2}\right).
\end{align*}
The second inequality follows from
\begin{align*}
\frac{5}{576n^2}+\frac{\left(-\frac{7}{72n}+\frac{5}{576n^2}\right)^2}{2}&=\frac{139}{10368n^2}-\frac{35}{41472n^3}+\frac{25}{663552n^4}<\frac{139}{10368n^2}<\frac{1}{50n^2}
\end{align*}
with the equality following from Maple.\par
A similar argument can be used for the opposite inequality.
\end{proof}
\begin{cor}
\label{cor:3.1}
For all $n\in\mathbb{N}$, $n\geq 100$, we have
\begin{equation*}
\frac{6.75^n}{3\sqrt{3\pi}n^{3/2}}\left(1+\frac{17}{72n}+\frac{3}{40n^2}\right)<S(n)<\frac{6.75^n}{3\sqrt{3\pi}n^{3/2}}\left(1+\frac{17}{72n}+\frac{1}{10n^2}\right).
\end{equation*}
and
\begin{equation*}
\frac{\sqrt{3}\cdot 6.75^n}{4\sqrt{\pi}n^{3/2}}\left(1-\frac{43}{72n}+\frac{1}{4n^2}\right)<B(n)<\frac{\sqrt{3}\cdot 6.75^n}{4\sqrt{\pi}n^{3/2}}\left(1-\frac{43}{72n}+\frac{1}{3n^2}\right).
\end{equation*}
\end{cor}
\begin{proof}
 We can deduce our bounds from equation \eqref{eqn3} and Propositions \ref{prop:S} and \ref{prop:binomialinequality}.
\end{proof}

Recall that $A(n)$ is the number of pairs $(1,1)$ pairs at depth $3n$ in the tree where there 
    are no restrictions on the walk.
Here we prove a weak bound for $A(n)$, which we use to derive that
\begin{equation*}
\lim_{n\rightarrow\infty}\frac{A(n+1)}{A(n)}=6.75.
\end{equation*}
\begin{prop}
\label{prop:Ainequality}
For all $n\in\mathbb{N}\cup\{0\}$, we have
\begin{equation*}
A(n)<2\cdot{3n\choose n}.
\end{equation*}
\end{prop}
\begin{proof}
We will prove by induction on $n$. One can check that this holds for $n=0,1,\dots,4$. Suppose for some $n\geq 5$, we have for all $0\leq i\leq n-1$ that
\begin{equation*}
A(i)\leq 2{3i\choose i}.
\end{equation*}
Then by \eqref{Aformula} we have
\begin{equation*}
A(n)<2\sum_{i=1}^n{3n-3i\choose n-i}S(i).
\end{equation*}
Noticing that for $S(n) = \frac{2}{3 n -1}{3n -1 \choose n-1}$ for $n \geq 2$ and 
    $S(1) = 5 = \frac{2}{3\cdot 1 - 1}{3 \cdot 1 - 1\choose 1 - 1} + 4$, we observe that 
\begin{equation*} 
\frac{2\sum_{i=1}^n{3n-3i\choose n-i}S(i)}{2 {3 n \choose n}} = \frac{8 {3 n -3\choose n-1} + 2\sum_{i=1}^n{3n-3i\choose n-i}\frac{2}{3 i - 1}{3 i - 1 \choose i -1}}{2 {3 n \choose n}}.
\end{equation*}
We can use Maple to evaluate the sum as in the proof of Proposition \ref{prop:S} and obtain that
\begin{equation*}
\frac{8 {3 n -3\choose n-1} + 2\sum_{i=1}^n{3n-3i\choose n-i}\frac{2}{3 i - 1}{3 i - 1 \choose i -1}}{2 {3 n \choose n}}=\frac{25 n^2 - 17 n + 2}{27 n^2 -27 n + 6}.
\end{equation*}
We observe that this is less than $1$ for all $n \geq 5$, proving 
    \[ A(n)<2\sum_{i=1}^n{3n-3i\choose n-i}S(i) < 2 {3 n \choose n} \]
as desired.  
\end{proof}
\begin{cor}
\label{cor:4.1}
We have
\begin{equation*}
A(n)\leq(1+o(1))\frac{\sqrt{3}\cdot 6.75^n}{\sqrt{\pi n}}.
\end{equation*}
\end{cor}
\begin{proof}
This follows from Propositions \ref{prop:binomialinequality} and \ref{prop:Ainequality}.
\end{proof}
\begin{lemma}
\label{lem:4.1}
For all $n\in\mathbb{N}$, $n\geq 2$ we have
\begin{equation*}
\frac{S(n+1)}{S(n)}<\frac{S(n+2)}{S(n+1)}.
\end{equation*}
\end{lemma}
\begin{proof}
By Proposition \ref{prop:S}, we have for each $n\geq 2$
\begin{equation*}
\frac{S(n+2)S(n)}{(S(n+1))^2}=\frac{36n^4+126n^3+158n^2+84n+16}{36n^4+126n^3+104n^2-14n-12}>1,
\end{equation*}
from which the result follows.
\end{proof}
\begin{lemma}
\label{lem:4.2}
For all $n\in\mathbb{N}$, we have
\begin{equation*}
\frac{A(n+1)}{A(n)}<\frac{A(n+2)}{A(n+1)}.
\end{equation*}
\end{lemma}
\begin{proof}
We prove by induction on $n$. First, for $n=1$, we have
\begin{equation*}
\frac{A(2)}{A(1)}=\frac{27}{5}<\frac{152}{27}=\frac{A(3)}{A(2)}.
\end{equation*}
Suppose by strong induction, we have
\begin{equation*}
\frac{A(i+1)}{A(i)}<\frac{A(i+2)}{A(i+1)}
\end{equation*}
for all $1\leq i\leq n-1$. Then we can deduce that
\begin{equation*}
\frac{A(i)}{A(n)}>\frac{A(i+1)}{A(n+1)}
\end{equation*}
for all $1\leq i\leq n-1$. Also, from \eqref{Aformula} we have:
\begin{align*}
\frac{A(n+1)}{A(n)}=&5+\frac{S(n+1)}{S(n)}-\frac{S(n+1)(A(n)-S(n))}{A(n)S(n)}+\frac{A(n+1)-5\cdot A(n)-S(n+1)}{A(n)}\\
=&5+\frac{S(n+1)}{S(n)}-\frac{S(n+1)}{A(n)S(n)}\left(\sum_{i=1}^{n-1}A(i)S(n-i)\right)+\frac{1}{A(n)}\left(\sum_{i=1}^{n-1}A(i)S(n+1-i)\right)\\
\end{align*}
By Lemma \ref{lem:4.1}, we can derive that
\begin{equation*}
\frac{S(n+1)}{S(n)}>\frac{S(n+1-i)}{S(n-i)}                                    	
\end{equation*}
or
\begin{equation*}
\frac{S(n+1)}{S(n)}\cdot S(n-i)-S(n+1-i)>0
\end{equation*}
for all $1\leq i<n$. Thus we have the following:
\begin{align*}
\frac{A(n+1)}{A(n)}=&5+\frac{S(n+1)}{S(n)}-\sum_{i=1}^{n-1}\frac{A(i)}{A(n)}\left(\frac{S(n+1)}{S(n)}\cdot S(n-i)-S(n+1-i)\right)\\
\frac{A(n+1)}{A(n)}<&5+\frac{S(n+1)}{S(n)}-\sum_{i=1}^{n-1}\frac{A(i+1)}{A(n+1)}\left(\frac{S(n+1)}{S(n)}\cdot S(n-i)-S(n+1-i)\right)\\
=&5+\frac{S(n+1)}{S(n)}-\frac{S(n+1)}{A(n+1)S(n)}\left(\sum_{i=1}^{n-1}A(i+1)S(n-i)\right)\\
&+\frac{1}{A(n+1)}\left(\sum_{i=1}^{n-1}A(i+1)S(n+1-i)\right)\\
=&5+\frac{S(n+1)}{S(n)}-\frac{S(n+1)(A(n+1)-5\cdot S(n)-S(n+1))}{A(n+1)S(n)}\\
&+\frac{A(n+2)-5\cdot A(n+1)-5\cdot S(n+1)-S(n+2)}{A(n+1)}\\
=&\frac{S(n+1)^2}{A(n+1)S(n)}+\frac{A(n+2)}{A(n+1)}-\frac{S(n+2)}{A(n+1)}\\
=&\frac{A(n+2)}{A(n+1)}-\frac{S(n+1)}{A(n+1)}\left(\frac{S(n+2)}{S(n+1)}-\frac{S(n+1)}{S(n)}\right)\\
<&\frac{A(n+2)}{A(n+1)}.
\end{align*}
The last inequality follows from Lemma \ref{lem:4.1}. Thus, by strong induction, we have our result.
\end{proof}
\begin{proof}[Proof of Proposition \ref{prop:Aratio}]
By Lemma \ref{lem:4.2}, we have $\frac{A(n+1)}{A(n)}$ is an increasing sequence in $n\in\mathbb{N}\cup\{0\}$ so the desired limit exists. Corollary \ref{cor:4.1} implies this limit is at most $6.75$, while Corollary \ref{cor:3.1} implies it is at least $6.75$.
\end{proof}

\section{Preliminary Results Concerning Other Coprime Pairs}
\label{app:coprime}

Rittaud \cite{average} constructed a subtree $\mathbf{R}$ from the Fibonacci tree consisting of all the shortest walks from the root $(1,1)$ down to each coprime pair $(a,b)$, calling it the restricted tree. The top part of of this subtree is shown in Figure \ref{fig:R}.

\begin{figure}[h]
\begin{center}
\begin{tikzpicture}
[level distance=1cm,
  level 1/.style={sibling distance=1cm},
  level 2/.style={sibling distance=1cm},
  level 3/.style={sibling distance=6.0cm},
  level 4/.style={sibling distance=4.0cm},
  level 5/.style={sibling distance=2.0cm},
  level 6/.style={sibling distance=1.2cm},
  level 7/.style={sibling distance=0.70cm},
  level 8/.style={sibling distance=0.40cm}]
\node{1}
  child{node {1}
    child {node {2}
      child{node {1}
        child{node {3}
          child{node{2}
            child{node{5}
              child{node{3}
                child{node{8}}
              }
              child{node{7}
                child{node{2}}
                child{node{12}}
              }
            }
          }
          child {node {4}
            child{node{1}
              child{node{5}
                child{node{4}}
                child{node{6}}
              }
            }
            child{node{7}
              child{node{3}
                child{node{10}}
              }
              child{node{11}
                child{node{4}}
                child{node{18}}
              }
            }
          }
        }
      }
      child{node {3}
        child{node{1}
          child{node{5}
            child{node{4}
              child{node{9}
                child{node{5}}
                child{node{13}}
              }
            }
          child{node{6}
            child{node{1}
               child{node{7}}
            }
            child{node{11}
               child{node{5}}
               child{node{17}}
             }
           }
         }
       }
       child {node {5}
         child{node{2}
           child{node{7}
             child{node{5}
               child{node{12}}
             }
             child{node{9}
               child{node{2}}
               child{node{16}}
             }
           }
         }
       child{node{8}
         child{node{3}
           child{node{11}
             child{node{8}}
             child{node{14}}
           }
          }
          child{node{13}
             child{node{5}
                child{node{18}}
             }
             child{node{21}
                child{node{8}}
                child{node{34}}
             }
          }
      }
    }
}
    }};
\end{tikzpicture}
\end{center}
\caption{The Restricted Tree $\mathbf{R}=\mathbf{R}_{(1,1)}$}
\label{fig:R}
\end{figure}
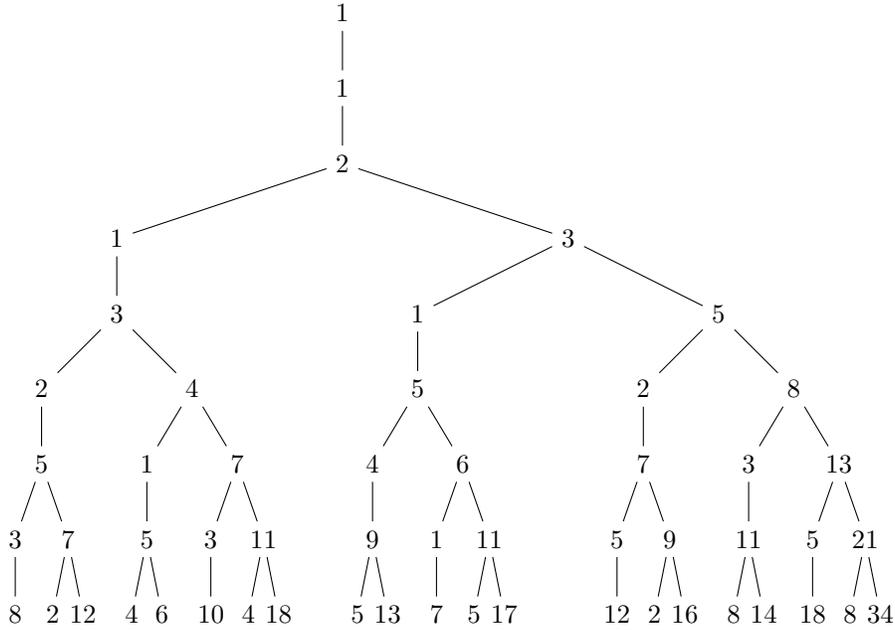
He proves the following in \cite{average}:
\begin{lemma}[Rittaud]
\label{lem:Rittaud}
The restricted tree $\mathbf{R}$ consists of all walks that do not have two left branches occurring with no right branch between them. Therefore, for all coprime pairs $(a,b)$, the shortest walk from the root $(1,1)$ to $(a,b)$ does not have two left branches occurring with no right branch between them.
\end{lemma}
\begin{lemma}
\label{lem:5.2}
Let the first occurrence of the coprime pair $(a,b)$ be at depth $k$. For any integer $n\geq 0$, there exists a walk with $k+3n$ branches that ends at a pair $(a,b)$. Moreover, if a walk of length $\ell$ ends at a $(a,b)$ pair, then $\ell-k\in 3\mathbb{N}$.
\end{lemma}
\begin{proof}
If $1,1,a_3,\dots,a_{k-2},a,b$ is a walk of length $k$, then $1,1,0,1,1,a_3,\dots,a_{k-2},a,b$ is a walk of length $k+3$. Hence the first part follows by induction, whereas the second follows by the parity of $a$ and $b$ and the minimality of $k$.
\end{proof}
\begin{prop}
\label{prop:coprimepair}
Take a coprime pair $(a,b)$ that isn't $(1,1)$ and let $1, 1, a_1, a_2, \dots, a_m, a, b$ be a walk from $(1,1)$ to $(a,b)$. Then $|a-b|,a,b$ occurs within this walk. 
\end{prop}
Note that the terminal pair $(a,b)$ may not be the only occurrence of the pair 
    $(a,b)$ that the walk traverses.
\begin{proof}
Suppose $SW_{(1,1)}(a,b)$. By Lemma \ref{lem:5.2}, the length of all of the possible walks are $k+3n$ where $n\in\mathbb{N}\cup\{0\}$. We will prove this by induction on $n$.

For $n=0$, we obtain the shortest walk from the root $(1,1)$ to $(a,b)$. By \cite[Corollary 5.1]{average} we have that the parent of $a$ of the ending pair $(a,b)$ is $|a-b|$.

Suppose now the proposition holds for all $0\leq n<N$ for some $N\in\mathbb{N}$. Take a walk from $(1,1)$ to $(a,b)$ consisting of $k+3N$ branches. If the first branch is a left branch, then we will attain another $(1,1)$ pair at depth $3$ in the tree and so we can remove these first three branches to obtain a walk of length $k+3(N-1)$ from which by induction the walk must consist of a pair $(a,b)$ such that the parent of this specific $a$ is $|a-b|$. Since we only removed the first three branches of the original walk, the original walk must have this property too.

Suppose that the walk in question starts with a right branch. We know that this walk isn't the shortest walk since $N\geq 1$. Therefore, by Lemma \ref{lem:Rittaud}, we must have that the walk consists of two left branches with no right branches between them. Since the first branch is a right branch, it therefore follows that somewhere in the tree we have a consecutive sequence of $3$ branches consisting of a right branch followed by two left branches. Suppose the branch immediately before this right branch (in case this specific right branch is the first branch in the walk consider the root $(1,1)$ here) consists of the pair $(c,d)$. Then taking the right branch and then the two left branches gives us the sequence (starting with the $(c,d)$ pair)
\begin{equation*}
c,d,c+d,c,d
\end{equation*}
Therefore the second left branch also consists of the pair $(c,d)$. Removing the right branch and the two left branches therefore gives us a shorter walk to the pair $(a,b)$. Since by induction this shorter walk must have a pair $(a,b)$ with this specific $a$ having a parent of $|a-b|$ in the walk, we therefore obtain that the original walk has this property too. By induction we obtain our result.
\end{proof}
\begin{lemma}
\label{lem:5.3}
Take the tree $\mathbf{T}_{(a,b)}$ for some coprime pair $(a,b)$ where $a,b\geq 0$. Suppose we take a finite walk in the tree, starting at the root and consisting of exactly twice as many left branches as right branches, but such that at any given intermediate point the number of left branches taken is less than or equal to twice the number of right branches taken. Then the pair on the last branch will be $(a,b)$.
\end{lemma}
\begin{proof}
Given a path as in the lemma, if at all intermediate points the number of left branches taken is strictly less than twice the number of right branches taken, then we can repeat the argument given in the proof of Proposition \ref{prop:primitive} to deduce that the ending pair will be $(a,b)$. For the broader collection of paths given in the lemma, we may then apply induction on the number of places in the given path where the number of left branches taken is exactly twice the number of right branches taken to obtain the result.
\end{proof}
\begin{lemma}
\label{lem:5.4}
Take a walk in $\mathbf{T}_{(a,b)}$ that starts at the root $(a,b)$ where $a,b\geq 0$ and $\gcd(a,b)=1$. Suppose that the number of left branches is strictly less than twice the number of right branches in this walk. Also suppose that at any given intermediate point in the walk the number of left branches taken is less than or equal to twice the number of right branches taken. Then the pair on the final branch will not be $(a,b)$.
\end{lemma}
\begin{proof}
Take such a path as described in the lemma. It is possible to extend this path by a series of left branches to obtain a path as described in Lemma \ref{lem:5.3} and hence the final pair on this extended path has to be $(a,b)$. Since the values of the nodes are decreasing along this series of left branches, we must have that the ending pair of the original path cannot be $(a,b)$.
\end{proof}
\begin{lemma}
\label{lem:5.5}
Take the tree $\mathbf{T}_{(a,b)}$ for some coprime pair $(a,b)$ and let $n\in\mathbb{N}$. Suppose we take all $(a,b)$ pairs at depth $3n$ in the tree such that the walks to these $(a,b)$ pairs satisfies the following. Let the first branch be a right branch  and the branch after any intermediate pair $(a,b)$ in the walk be a right branch. The number of such $(a,b)$ pairs is $B(n)$.
\end{lemma}
\begin{proof}
We will first show that the walks in question are characterised as follows. There are twice as many left branches as right branches and at any given intermediate point the number of left branches encountered is less than or equal to twice the number of right branches encountered. A walk characterised as such will begin with a right branch. Moreover, at the first point, whether it be some intermediate point or at the final branch, the number of left branches will stop being less than twice the number of right branches and will instead be equal to it. By Lemma \ref{lem:5.4}, the pairing we encounter at this branch will be $(a,b)$. If this is an intermediate point, then we must take a right branch to preserve the inequality. This will continue on until we come to the last branch that will also have the pair $(a,b)$. Thus such a walk will satisfy the criteria in this lemma.

Conversely, a walk described as in this lemma begins with a right branch and when it attains a $(a,b)$ pair again, we must have twice as many left branches as right branches by Lemmas \ref{lem:5.3} and \ref{lem:5.4}. Then we take another right branch and so on. This fits the characterisation we have given. Thus it has become a question of counting the number of walks that are characterised as in the start of the proof. By using the definitions of $S(n)$ and $B(n)$ and Proposition \ref{prop:primitive}, we can see that this is $B(n)$.
\end{proof}
\begin{lemma}
\label{lem:5.6}
Let $a,b\geq 0$ with $\gcd(a,b)=1$. Consider a walk in $\mathbf{T}_{(a,b)}$ that starts with a left branch and ends at a $(a,b)$ pair with no intermediate $(a,b)$ pair. The parent of $a$ in the last pair is $|a-b|$.
\end{lemma}
\begin{proof}
We prove our result by induction on $n$ where $3n$ is the length of the walk in question. For $n=1$, we have the sequence
\begin{equation*}
a,b,|a-b|,a,b.
\end{equation*}
Suppose it holds for some $n\in\mathbb{N}$. We want to show it holds for $n+1$. So consider a walk of length $3n+3$ that starts at the root $(a,b)$ where the first branch is a left branch and ends at a pair $(a,b)$. We wish to show that the third last term in the sequence is $|a-b|$. Suppose for a contradiction it isn't. Then the third last term must $a+b$. Since $b=|a-(a+b)|$, the final branch must be a left branch. Also since $a<a+b$, the second last branch must also be a left branch. Thus somewhere in the walk there must be a right branch immediately followed by two left branches. As in the proofs of Proposition \ref{prop:coprimepair} and Lemma \ref{lem:5.3} such a configuration can be dropped out without affecting the pairing on the last branch $(a,b)$. But then this smaller walk would not have any intermediate $(a,b)$ pairs and the third last term would still be $a+b$, which isn't possible by our inductive assumption. Therefore, the third last term of the original walk had to have been $|a-b|$ as well. Thus we have our result.
\end{proof}
\begin{cor}
\label{cor:5.1}
Take the tree $\mathbf{T}_{(|a-b|,a)}$ for some coprime pair $(a,b)$. Then for all walks to an $(a,b)$, there must exist a pair $(a,b)$ in the walk such that the parent of that specific $a$ is $|a-b|$ in the walk.
\end{cor}
\begin{proof}
Take such a walk to a pair $(a,b)$ and suppose there exists no pair $(a,b)$ in that walk such that the parent of that specific $a$ in the walk is $|a-b|$. Suppose we lengthen the walk in front by adding a node $b$ to be the parent of $|a-b|$ and then another node $a$ to be the parent of $b$. This will give a walk that starts with two left branches if $a<b$ or a walk that starts with a left branch and then a right branch if $a\geq b$. In either case, we have a walk that contradicts Lemma \ref{lem:5.6}. 

Therefore the result follows.
\end{proof}
\begin{prop}
\label{prop:coprimepair2}
Let $a$ and $b$ be coprime integers and $n\in\mathbb{N}$. In $\mathbf{T}_{(a,b)}$, the number of $(a,b)$ pairs at depth $3n$ in the tree that can be attained by a walk not containing a $(|a-b|,a)$ pair is equal to $B(n)$.
\end{prop}
\begin{proof}
Combine Lemmas \ref{lem:5.5} and \ref{lem:5.6} and Corollary \ref{cor:5.1}.
\end{proof}

\begin{proof}[Proof of Proposition \ref{prop:coprimepair3}]
By Proposition \ref{prop:coprimepair} any walk in the Fibonacci tree that starts at the root $(1,1)$ and ends at the pair $(a,b)$ must contain the pair $(|a-b|,a)$. Consider the last place in a given walk that this pair occurs and say it is at depth $3i+m$ in the tree where $0\leq m\leq 2$ (where $m$ depends on the parity of $(a,b)$). Then by Corollary \ref{cor:5.1} the next element in the walk is $b$ and, by Proposition \ref{prop:coprimepair2}, this gives rise to $B(n-i)$ pairs of $(a,b)$ at depth $3n+m$ or $3(n+1)+m$ in the tree (depending on the parity of $a$ and $b$). Conversely, every pair $(|a-b|,a)$ that occurs at an intermediate point at depth $3i+m$ in the tree gives rise to $B(n-i)$ walks to pairs of $(a,b)$ at depth $3n+m$ or $3(n+1)+m$ in the tree. The summation starts at $i=\lfloor\frac{k-1}{3}\rfloor$ since $SW_{(1,1)}(|a-b|,a)=k-1$ and so the pair $(|a-b|,a)$ occurs at depth $3\lfloor\frac{k-1}{3}\rfloor+m$ in the tree. Thus the formula follows.
\end{proof}

\begin{proof}[Proof of Lemma \ref{lem:5.7}]
We prove by induction on $SW_{(1,1)}(a,b)$. First, suppose $SW_{(1,1)}(a,b)=3$. Then both $a$ and $b$ are odd. By Proposition \ref{prop:coprimepair3}, we have
\begin{equation*}
A_{(a,b)}(n)=\sum_{i=0}^{n-1}A_{(|a-b|,a)}(i)B(n-1-i)
\end{equation*}
for all $n\in\mathbb{N}$ where $(|a-b|,a)$ is a pair satisfying $SW_{(1,1)}(|a-b|,a)=2$. There are only two pairs that $(|a-b|,a)$ can be: $(2,1)$ or $(2,3)$. Thus by Corollary \ref{cor:5.4}, we have
\begin{align*}
A_{(a,b)}(n)&=\sum_{i=0}^{n-1}(A_{(1,1)}(i+1)-4\cdot A_{(1,1)}(i))B(n-1-i)\\
&=\sum_{i=0}^{n-1}A_{(1,1)}(i+1)B(n-1-i)-4\sum_{i=0}^{n-1}A_{(1,1)}(i)B(n-1-i)
\end{align*}
By Corollary \ref{cor:5.3} and \eqref{ABformula}, we have
\begin{align*}
A_{(a,b)}(n)&=A_{(1,2)}(n)-B(n)-4\sum_{i=0}^{n-1}A_{(1,1)}(i)B(n-1-i)\\
&=A_{(1,2)}(n)-A_{(1,1)}(n).
\end{align*}
Thus it holds for all $n\in\mathbb{N}$ for the pair $(a,b)$ since $SW_{(1,1)}(1,2)=1$ and $SW_{(1,1)}(1,1)=0$. Suppose the proposition holds for pairs that have a shortest walk of length $k-1$ branches for some $k\geq 4$ and suppose we want to show it holds for pairs with shortest walks of lengths $k$. Let $(a,b)$ be a pair with $SW_{(1,1)}(a,b)=k$. Let the last six elements of the shortest walk to $(a,b)$ be
\begin{equation*}
a_0,a_1,a_2,|a-b|,a,b.
\end{equation*}
First, suppose that both $a$ and $b$ are odd. Then, by Proposition \ref{prop:coprimepair3}, we have
\begin{equation*}
A_{(a,b)}(n)=\sum_{i=\lfloor\frac{k-1}{3}\rfloor}^{n-1}A_{(|a-b|,a)}(i)B(n-1-i)
\end{equation*}
for all $n\geq\frac{k}{3}$. By our inductive hypothesis, we have
\begin{equation*}
A_{(|a-b|,a)}(i)=A_{(a_1,a_2)}(i+1)-A_{(a_0,a_1)}(i)
\end{equation*}
for all $i\geq\lfloor\frac{k-1}{3}\rfloor$ since $|a-b|$ is even and $a$ is odd. Then we have
\begin{align*}
A_{(a,b)}(n)&=\sum_{i=\lfloor\frac{k-1}{3}\rfloor}^{n-1}(A_{(a_1,a_2)}(i+1)-A_{(a_0,a_1)}(i)
)B(n-1-i)\\
&=\sum_{i=\lfloor\frac{k-1}{3}\rfloor}^{n-1}A_{(a_1,a_2)}(i+1)B(n-1-i)-\sum_{i=\lfloor\frac{k-1}{3}\rfloor}^{n-1}A_{(a_0,a_1)}(i)B(n-1-i)\\
&=\sum_{i=\lfloor\frac{k-1}{3}\rfloor+1}^{n}A_{(a_1,a_2)}(i)B(n-i)-\sum_{i=\lfloor\frac{k-1}{3}\rfloor}^{n-1}A_{(a_0,a_1)}(i)B(n-1-i)\\
&=\sum_{i=\lfloor\frac{k-3}{3}\rfloor+1}^{n}A_{(a_1,a_2)}(i)B(n-i)-\sum_{i=\lfloor\frac{k-4}{3}\rfloor+1}^{n-1}A_{(a_0,a_1)}(i)B(n-1-i)
\end{align*}
Thus, by Proposition \ref{prop:coprimepair3}, we have
\begin{align*}
A_{(a,b)}(n)=A_{(a_2,|a-b|)}(n)-B\left(n-\left\lfloor\frac{k-3}{3}\right\rfloor\right)-\left(A_{(a_1,a_2)}(n)-B\left(n-1-\left\lfloor\frac{k-4}{3}\right\rfloor\right)\right)
\end{align*}
since $SW_{(1,1)}(a_2,|a-b|)=k-2$ and $SW_{(1,1)}(a_1,a_2)=k-3$. Also, we have $\lfloor\frac{k-3}{3}\rfloor=1+\lfloor\frac{k-4}{3}\rfloor$ since $3|k$. Thus we get our result
\begin{equation*}
A_{(a,b)}(n)=A_{(a_2,|a-b|)}(n)-A_{(a_1,a_2)}(n)
\end{equation*}
for all $n\geq\frac{k}{3}$.

By a similar argument, if $a$ is odd and $b$ is even, then
\begin{equation*}
A_{(a,b)}(n)=A_{(a_2,|a-b|)}(n)-A_{(a_1,a_2)}(n).
\end{equation*}
Also, by a similar argument, if $a$ is even and $b$ is odd, then
\begin{equation*}
A_{(a,b)}(n)=A_{(a_2,|a-b|)}(n+1)-A_{(a_1,a_2)}(n).
\end{equation*}
\end{proof}

\section{Details of Proof of Proposition \ref{prop:Dinequality}}

\label{app:D}
We establish our asymptotic results concerning $A(n)=A_{1,1}(n)$ here. First, we prove a couple of lemmas:

\begin{lemma}
\label{lem:6.0}
For all $n\in\mathbb{N}$, we have
\begin{equation*}
\sum_{k=0}^{n}B(k)B(n-k)=S(n+1).
\end{equation*}
\end{lemma}
\begin{proof}
In \cite{identity} we have the combinatorial identity
\begin{equation*}
\sum_{k=0}^n{tk+r\choose k}{tn-tk+s\choose n-k}\frac{r}{tk+r}\cdot\frac{s}{tn-tk+s}={tn+r+s\choose n}\frac{r+s}{tn+r+s}
\end{equation*}
valid for all $n\in\mathbb{N}$ and all $r,s,t\in\mathbb{R}$. Substituting in $t=3$, $r=1$, and $s=1$ gives us
\begin{equation}
\sum_{k=0}^n{3k+1\choose k}{3n-3k+1\choose n-k}\frac{1}{3k+1}\cdot\frac{1}{3n-3k+1}={3n+2\choose n}\frac{2}{3n+2}\label{eqn5}.
\end{equation}
Also for all $n\in\mathbb{N}\cup\{0\}$, we have
\begin{equation*}
B(n)=\frac{1}{2n+1}{3n\choose n}=\frac{1}{2n+1}\frac{(3n)!}{n!(2n)!}=\frac{(3n)!}{n!(2n+1)!}=\frac{1}{3n+1}{3n+1\choose n}.
\end{equation*}
Thus by \eqref{eqn5} and Proposition \ref{prop:S} we have for all $n\in\mathbb{N}$
\begin{equation*}
\sum_{k=0}^{n}B(k)B(n-k)=S(n+1).
\end{equation*}
\end{proof}

\begin{proof}[Proof of Lemma \ref{lem:6.1}]
By Proposition \ref{prop:coprimepair3}, we have, for all $n\in\mathbb{N}\cup\{0\}$
\begin{equation*}
A_{(2,1)}(n)=\sum_{i=0}^{n}A_{(1,2)}(i)B(n-i).
\end{equation*}
By Corollaries \ref{cor:5.2} and \ref{cor:5.4}, we therefore have
\begin{equation*}
A(n+1)-4\cdot A(n)=\sum_{i=0}^{n}\frac{A(i+1)-B(i+1)}{4}B(n-i).
\end{equation*}
Thus we have the following:
\begin{align*}
A(n+1)-4\cdot A(n)
&=\frac{1}{4}\sum_{i=1}^{n+1}A(i)B(n+1-i)-\frac{1}{4}\sum_{i=0}^{n}B(i+1)B(n-i).
\end{align*}
By \eqref{ABformula}, we thus have
\begin{align}
A(n+1)-4\cdot A(n)&=\frac{1}{4}\left(\frac{A(n+2)-B(n+2)}{4}-B(n+1)\right)-\frac{1}{4}\sum_{i=0}^{n}B(i+1)B(n-i)\nonumber\\
&=\frac{A(n+2)-B(n+2)}{16}-\frac{B(n+1)}{4}-\frac{1}{4}\sum_{i=0}^{n}B(i+1)B(n-i)\label{eqn4}.
\end{align}
By Lemma \ref{lem:6.0}, we have
\begin{equation*}
\sum_{i=0}^nB(i+1)B(n-i)=\sum_{i=1}^{n+1}B(i)B(n+1-i)=S(n+2)-B(n+1).
\end{equation*}
Thus by \eqref{eqn4} we have for all $n\in\mathbb{N}$
\begin{align*}
A(n+1)-4\cdot A(n)&=\frac{A(n+2)-B(n+2)}{16}-\frac{B(n+1)}{4}-\frac{1}{4}(S(n+2)-B(n+1))\\
&=\frac{A(n+2)-B(n+2)}{16}-\frac{S(n+2)}{4}.
\end{align*}
Thus, for all $n\in\mathbb{N}$, we have our result.
\end{proof}

\begin{proof}[Proof of Proposition \ref{prop:A}]
We can derive this from Propositions \ref{prop:S} and \ref{prop:Aratio}, Corollary \ref{cor:3.1}, and Lemma \ref{lem:6.1}.
\end{proof}
\begin{note}
For the rest of this section let
\begin{equation*}
f(n):=\frac{25\cdot 6.75^n}{12\sqrt{3\pi}n^{3/2}}.
\end{equation*}
\end{note}
\begin{prop}
\label{prop:Binequality}
For all $n\in\mathbb{N}$, $n\geq 100$, we have
\begin{equation*}
f(n+2)\left(1-\frac{23}{360(n+2)}+\frac{69}{500(n+2)^2}\right)<B(n+2)+4\cdot S(n+2)
\end{equation*}
and
\begin{equation*}
B(n+2)+4\cdot S(n+2)<f(n+2)\left(1-\frac{23}{360(n+2)}+\frac{23}{125(n+2)^2}\right).
\end{equation*}
\end{prop}
\begin{proof}
We can get our bounds from Corollary \ref{cor:3.1}.
\end{proof}

\begin{proof}[Proof of the upper bound for $D(n)$]
Note that
\begin{equation*}
-\frac{405\sqrt{3}\cdot 6.75^n}{16\sqrt{\pi}n^{3/2}}=-\frac{729f(n)}{20}.
\end{equation*}
Suppose for a contradiction that for some $n\geq 100$, we have
\begin{align*}
D(n)\geq -\frac{729f(n)}{20}\left(1-\frac{4019}{360n}\right).
\end{align*}
The right-hand side of the above inequality is a transcendental number for all $n\in\mathbb{N}$, and since $D(n)\in\mathbb{Z}$ for all $n\in\mathbb{N}$, we must therefore have that
\begin{align*}
D(n)>-\frac{729f(n)}{20}\left(1-\frac{4019}{360}\right).
\end{align*}
Thus we have
\begin{equation*}
D(n)=-\frac{729f(n)}{20}\left(1-\frac{C}{n}\right).
\end{equation*}
where $C>\frac{4019}{360}$. Then, by Proposition \ref{prop:Binequality}, we have
\begin{align*}
D(n+1)=&8\cdot D(n)+B(n+2)+4\cdot S(n+2)\\
>&\left(-\frac{405\sqrt{3}\cdot 6.75^n}{2\sqrt{\pi}(n+1)^{3/2}}\right)\left(\frac{n+1}{n}\right)^{3/2}\left(1-\frac{C}{n}\right)\\
&+\frac{25\cdot 6.75^{n+2}}{12\sqrt{3\pi}(n+1)^{3/2}}\left(\frac{n+1}{n+2}\right)^{3/2}\left(1-\frac{23}{360(n+2)}\right).
\end{align*}
By obtaining good enough bounds for $\left(\frac{n+1}{n}\right)^{3/2}$ and $\left(\frac{n+1}{n+2}\right)^{3/2}$ using the binomial theorem expansion of $(1+x)^{-3/2}$, we can deduce that following bound (see Appendix \ref{app:Dinequality}):
\begin{equation*}
D(n+1)>\left(-\frac{405\sqrt{3}\cdot 6.75^n}{2\sqrt{\pi}(n+1)^{3/2}}\right)\left(1+\frac{3-2C}{2(n+1)}\right)+\left(\frac{25\cdot 6.75^{n+2}}{12\sqrt{3\pi}(n+1)^{3/2}}\right)\left(1-\frac{563}{360(n+1)}\right).
\end{equation*}
Thus we have
\begin{align*}
D(n+1)&>-\frac{729f(n+1)}{20}\left(1-\frac{2304C-4019}{1944(n+1)}\right).
\end{align*}
We deduce
\begin{equation*}
r:=\frac{2304C-4019}{1944C}>1.
\end{equation*}
Repeating the argument with $C_1$ in place of $C$ and $n+1$ in place of $n$ gives us
\begin{equation*}
D(n+2)=\left(-\frac{405\sqrt{3}\cdot 6.75^{n+2}}{16\sqrt{\pi}(n+2)^{3/2}}\right)\left(1-\frac{C_2}{(n+2)}\right)
\end{equation*}
where $C_2>rC_1>r^2C$. Repeating the argument as many times as necessary, we thus have, for all $k\in\mathbb{N}$,
\begin{align*}
D(n+k)=\left(-\frac{405\sqrt{3}\cdot 6.75^{n+k}}{16\sqrt{\pi}(n+k)^{3/2}}\right)\left(1-\frac{C_k}{(n+k)}\right)
\end{align*}
where $C_k>r^kC$. This leads to
\begin{equation*}
\lim_{k\rightarrow\infty}\frac{r^k}{k}=0,
\end{equation*}
which doesn't hold since $r>1$, a contradiction. Thus we have our first desired inequality for all $n\geq 100$.
\end{proof}
\begin{proof}[Proof of the lower bound for $D(n)$]
Suppose for a contradiction that there exists  $n\geq 100$ such that
\begin{equation*}
D(n)=-\frac{729f(n)}{20}\left(1-\frac{4019}{360n}+\frac{C}{n}\right).
\end{equation*}
where $\frac{207}{n}\leq C$. Applying the same techniques as in the first inequality (see Appendix \ref{app:Dinequality}), we can obtain
\begin{align*}
D(n+1)<&-\frac{729f(n+1)}{20}\left(1-\frac{4019}{360(n+1)}+\frac{28C}{27(n+1)}\right).
\end{align*}
We can again repeat the argument as many times as necessary (see Appendix \ref{app:Dinequality}) to get that, for all $k\in\mathbb{N}$,
\begin{equation*}
D(n+k)<-\frac{729f(n+k)}{20}\left(1-\frac{4019}{360(n+k)}+\frac{\left(\frac{28}{27}\right)^kC}{(n+k)}\right).
\end{equation*}
This leads to
\begin{equation*}
\lim_{k\rightarrow\infty}\frac{\left(\frac{28}{27}\right)^k}{k}=0,
\end{equation*}
which doesn't hold, a contradiction. Thus we have our second desired inequality for all $n\geq 100$.
\end{proof}

\section{Algebra in Proof of Proposition \ref{prop:Dinequality}}
\label{app:Dinequality}

\subsection{Binomial Theorem Calculations}
We have
\begin{align}
\left(\frac{n+1}{n}\right)^{3/2}=&\left(\frac{n}{n+1}\right)^{-3/2}\nonumber\\
=&\left(1-\frac{1}{n+1}\right)^{-3/2}\label{eqn14}
\end{align}
For all $0<x\leq\frac{1}{101}$, we have, by the binomial theorem,
\begin{align}
(1-x)^{-3/2}&=1+\frac{3x}{2}+\frac{15x^2}{8}+\dots+\frac{\frac{3}{2}\cdot\frac{5}{2}\dots\frac{2k+1}{2}x^k}{k!}+\dots\label{eqn15}\\
&<1+\frac{3x}{2}+\frac{19x^2}{10}\label{eqn16}.
\end{align}
Let $f(x)=(x+1)^{-3/2}$ and $g(x)=1-\frac{3x}{2}$. We have $f(0)=g(0)=1$ and for all $x>0$, we have
\begin{equation*}
f'(x)=\frac{-3}{2}(x+1)^{-5/2}>\frac{-3}{2}=g'(x)
\end{equation*}
where $f'(x)$ and $g'(x)$ are the derivatives of $f(x)$ and $g(x)$ respectively. Thus for all $x>0$, we must have $f(x)>g(x)$ or
\begin{equation}
(x+1)^{-3/2}>1-\frac{3x}{2}.\label{eqn21}
\end{equation}
Thus we have
\begin{align*}
D(n+1)>&\left(-\frac{405\sqrt{3}\cdot 6.75^n}{2\sqrt{\pi}(n+1)^{3/2}}\right)\left(1+\frac{3}{2(n+1)}+\frac{19}{10(n+1)^2}\right)\left(1-\frac{C}{(n+1)}\right)\\
&+\frac{25\cdot 6.75^{n+2}}{12\sqrt{3\pi}(n+1)^{3/2}}\left(1-\frac{3}{2(n+1)}\right)\left(1-\frac{23}{360(n+1)}\right)\\
=&\left(-\frac{405\sqrt{3}\cdot 6.75^n}{2\sqrt{\pi}(n+1)^{3/2}}\right)\left(1+\frac{3-2C}{2(n+1)}+\frac{19-15C}{10(n+1)^2}-\frac{19C}{10(n+1)^3}\right)\\
&+\frac{25\cdot 6.75^{n+2}}{12\sqrt{3\pi}(n+1)^{3/2}}\left(1-\frac{563}{360(n+1)}+\frac{11}{48(n+1)^2}-\frac{1}{5(n+1)^3}\right)\\
>&\left(-\frac{405\sqrt{3}\cdot 6.75^n}{2\sqrt{\pi}(n+1)^{3/2}}\right)\left(1+\frac{3-2C}{2(n+1)}\right)\\
&+\frac{25\cdot 6.75^{n+2}}{12\sqrt{3\pi}(n+1)^{3/2}}\left(1-\frac{563}{360(n+1)}\right).
\end{align*}
\subsubsection{Proof of Second Inequality}
We can derive that
\begin{equation}
\frac{4019}{360n}<\frac{4019}{360(n+1)}+\frac{C}{32(n+1)}\label{eqn12}
\end{equation}
and
\begin{equation}
\frac{2228257}{97200(n+1)}<\frac{C}{9}.\label{eqn13}
\end{equation}
By Proposition \ref{prop:Binequality}, we have
\begin{align*}
D(n+1)<&\left(-\frac{405\sqrt{3}\cdot 6.75^n}{2\sqrt{\pi}(n+1)^{3/2}}\right)\left(\frac{n+1}{n}\right)^{3/2}\left(1-\frac{4019}{360n}+\frac{C}{n}\right)\\
&+\frac{25\cdot 6.75^{n+2}}{12\sqrt{3\pi}(n+1)^{3/2}}\left(\frac{n+1}{n+2}\right)^{3/2}\left(1-\frac{23}{360(n+2)}+\frac{23}{125(n+2)^2}\right).
\end{align*}
By \eqref{eqn14}, \eqref{eqn15}, and \eqref{eqn12}, we have
\begin{align*}
D(n+1)<&\left(-\frac{405\sqrt{3}\cdot 6.75^n}{2\sqrt{\pi}(n+1)^{3/2}}\right)\left(1+\frac{3}{2(n+1)}\right)\left(1-\frac{4019}{360(n+1)}+\frac{C}{(n+1)}-\frac{C}{32(n+1)}\right)\\
&+\frac{25\cdot 6.75^{n+2}}{12\sqrt{3\pi}(n+1)^{3/2}}\left(1+\frac{1}{(n+1)}\right)^{-3/2}\left(1-\frac{23}{360(n+1)}+\frac{23}{360(n+1)^2}+\frac{23}{125(n+1)^2}\right)\\
\end{align*}
For all $0<x<\frac{8}{9}$, we have, by the binomial theorem,
\begin{equation*}
(1+x)^{-3/2}=1-\frac{3x}{2}+\frac{15x^2}{8}-\dots+\frac{(-1)^k\frac{3}{2}\cdot\frac{5}{2}\dots\frac{2k+1}{k}x^k}{k!}.
\end{equation*}
For all $k\geq 3$, $k$ odd, we have
\begin{align*}
\frac{(-1)^k\frac{3}{2}\cdot\frac{5}{2}\dots\frac{2k+1}{k}x^k}{k!}+\frac{(-1)^{k+1}\frac{3}{2}\cdot\frac{5}{2}\dots\frac{2k+1}{k}x^{k+1}}{(k+1)!}&=-\frac{\frac{3}{2}\cdot\frac{5}{2}\dots\frac{2k+1}{2}x^k}{k!}+\frac{\frac{3}{2}\cdot\frac{5}{2}\dots\frac{2k+3}{2}x^{k+1}}{(k+1)!}\\
&=\frac{\frac{3}{2}\cdot\frac{5}{2}\dots\frac{2k+1}{k}x^k}{k!}\left(-1+\frac{(2k+3)x}{2k+2}\right)\\
&\leq\frac{\frac{3}{2}\cdot\frac{5}{2}\dots\frac{2k+1}{k}x^k}{k!}\left(-1+\frac{9x}{8}\right)\\
&<0.
\end{align*}
Thus
\begin{equation}
(1+x)^{-3/2}<1-\frac{3x}{2}+\frac{15x^2}{8}.\label{eqn20}
\end{equation}
Thus
\begin{align*}
D(n+1)<&\left(-\frac{405\sqrt{3}\cdot 6.75^n}{2\sqrt{\pi}(n+1)^{3/2}}\right)\left(1+\frac{3}{2(n+1)}\right)\left(1-\frac{4019}{360(n+1)}+\frac{C}{(n+1)}-\frac{C}{32(n+1)}\right)\\
&+\frac{25\cdot 6.75^{n+2}}{12\sqrt{3\pi}(n+1)^{3/2}}\left(1-\frac{3}{2(n+1)}+\frac{15}{8(n+1)^2}\right)\left(1-\frac{23}{360(n+1)}+\frac{23}{360(n+1)^2}+\frac{23}{125(n+1)^2}\right)\\
<&\left(-\frac{405\sqrt{3}\cdot 6.75^n}{2\sqrt{\pi}(n+1)^{3/2}}\right)\left(1-\frac{4019}{360(n+1)}+\frac{31C+48}{32(n+1)}+\frac{1395C-18236}{960(n+1)^2}\right)\\
&+\frac{25\cdot 6.75^{n+2}}{12\sqrt{3\pi}(n+1)^{3/2}}\left(1-\frac{563}{360(n+1)}+\frac{39937}{18000(n+1)^2}-\frac{3933}{8000(n+1)^3}+\frac{2231}{4800(n+1)^4}\right)\\
<&\left(-\frac{405\sqrt{3}\cdot 6.75^n}{2\sqrt{\pi}(n+1)^{3/2}}\right)\left(1-\frac{4019}{360(n+1)}+\frac{31C+48}{32(n+1)}+\frac{1395C-18236}{960(n+1)^2}\right)\\
&+\frac{25\cdot 6.75^{n+2}}{12\sqrt{3\pi}(n+1)^{3/2}}\left(1-\frac{563}{360(n+1)}+\frac{39937}{18000(n+1)^2}\right)\\
=&\left(-\frac{405\sqrt{3}\cdot 6.75^{n+1}}{16\sqrt{\pi}(n+1)^{3/2}}\right)\left(1-\frac{4019}{360(n+1)}+\frac{31C}{27(n+1)}+\frac{167400C-2228257}{97200(n+1)^2}\right).
\end{align*}
By \eqref{eqn13}, we have
\begin{align*}
D(n+1)<&\left(-\frac{405\sqrt{3}\cdot 6.75^{n+1}}{16\sqrt{\pi}(n+1)^{3/2}}\right)\left(1-\frac{4019}{360(n+1)}+\frac{31C}{27(n+1)}-\frac{2228257}{97200(n+1)^2}\right)\\
<&\left(-\frac{405\sqrt{3}\cdot 6.75^{n+1}}{16\sqrt{\pi}(n+1)^{3/2}}\right)\left(1-\frac{4019}{360(n+1)}+\frac{31C}{27(n+1)}-\frac{C}{9(n+1)}\right)\\
=&\left(-\frac{405\sqrt{3}\cdot 6.75^{n+1}}{16\sqrt{\pi}(n+1)^{3/2}}\right)\left(1-\frac{4019}{360(n+1)}+\frac{28C}{27(n+1)}\right).
\end{align*}
From the fact that $\frac{207}{n}\leq C$ we can deduce that $\frac{207}{(n+1)}\leq\frac{28C}{27}$. Thus we can repeat the above argument with $\frac{28C}{27}$ in place of $C$ and $n+1$ in place of $n$ to derive that
\begin{equation*}
D(n+2)<\left(-\frac{405\sqrt{3}\cdot 6.75^{n+2}}{16\sqrt{\pi}(n+2)^{3/2}}\right)\left(1-\frac{4019}{360(n+2)}+\frac{\left(\frac{28}{27}\right)^2C}{(n+2)}\right).
\end{equation*}
Repeating the argument as many times as necessary, we get that, for all $k\in\mathbb{N}$,
\begin{equation*}
D(n+k)<\left(-\frac{405\sqrt{3}\cdot 6.75^{n+k}}{16\sqrt{\pi}(n+k)^{3/2}}\right)\left(1-\frac{4019}{360(n+k)}+\frac{\left(\frac{10}{9}\right)^kC}{(n+k)}\right).
\end{equation*}
We know that
\begin{equation*}
\lim_{k\rightarrow\infty}\frac{-D(n+k)16\sqrt{\pi}(n+k)^{3/2}}{405\sqrt{3}\cdot 6.75^{n+k}}=1
\end{equation*}
so that
\begin{equation*}
\lim_{k\rightarrow\infty}\frac{-4019}{360(n+k)}-\frac{\left(\frac{28}{27}\right)^kC}{n+k}=0.
\end{equation*}
Thus
\begin{equation*}
\lim_{k\rightarrow\infty}\frac{\left(\frac{28}{27}\right)^k}{n+k}=0
\end{equation*}
so that
\begin{equation*}
\lim_{k\rightarrow\infty}\frac{\left(\frac{28}{27}\right)^k}{k}=0,
\end{equation*}
which doesn't hold, a contradiction. Thus we have our second desired inequality for all $n\geq 100$.

\section{Details of Proof of Theorem \ref{thm:Aab}}
\label{app:Aab}

We first need the following proposition:

\begin{prop}
\label{prop:A12}
For all $n\in\mathbb{N}$, we have
\begin{equation*}
\frac{405\cdot 6.75^n}{4\sqrt{3\pi}n^{3/2}}\left(1-\frac{60877}{2880n}+\frac{29}{n^2}\right)<A_{(1,2)}(n)<\frac{405\cdot 6.75^n}{4\sqrt{3\pi}n^{3/2}}\left(1-\frac{60877}{2880n}+\frac{669}{n^2}\right)
\end{equation*}
and
\begin{align*}
\frac{2673\cdot 6.75^n}{16\sqrt{3\pi}n^{3/2}}\left(1-\frac{18173}{792n}-\frac{16072}{99n^2}\right)& <A_{(2,1)}(n)\\ 
    & =A_{(2,3)}(n) \\
    & <\frac{2673\cdot 6.75^n}{16\sqrt{3\pi}n^{3/2}}\left(1-\frac{18173}{792n}+\frac{88768}{99n^2}\right).
\end{align*}
\end{prop}
\begin{proof}
From Corollaries \ref{cor:3.1} and \ref{cor:5.2} and theorem \ref{thm:A}, we can deduce the bounds for $A_{(1,2)}(n)$. By Corollary \ref{cor:5.4}, we have
\begin{equation*}
A_{(2,1)}=A_{(2,3)}=D(n)+4\cdot A_{(1,1)}(n).
\end{equation*}
Applying Proposition \ref{prop:Dinequality} and Theorem \ref{thm:A} gives us the desired bounds for $A_{(2,1)}(n)=A_{(2,3)}(n)$.
\end{proof}

We apply induction on $SW_{(1,1)}(a,b)$. Theorem \ref{thm:A} and Proposition \ref{prop:A12} provide the cases for $k=0$, $k=1$, and $k=2$. Suppose the result holds for pairs with shortest walks consisting of $k-2$ branches and $k-3$ branches for some $k\geq 3$ and we want to show it also holds for $k$. Let $(a,b)$ be a pair with $SW_{(1,1)}(a,b)=n$ and let the fourth last and third last branches in this walk have the pairs $(a_0,a_1)$ and $(a_1,|a-b|)$ respectively. By our inductive hypothesis, we have
\begin{equation*}
A_{(a_0,a_1)}=\frac{C_{(a_0,a_1)}\cdot 6.75^n}{n^{3/2}}\left(1+\frac{s_{k-3}}{n}+O\left(\frac{1}{n^2}\right)\right)
\end{equation*}
and
\begin{equation*}
A_{(a_1,|a-b|)}=\frac{C_{(a,|a-b|)}\cdot 6.75^n}{n^{3/2}}\left(1+\frac{s_{k-2}}{n}+O\left(\frac{1}{n^2}\right)\right)
\end{equation*}
where
\begin{equation*}
C_{(a_0,a_1)}=\frac{243\cdot t_{k-3}}{4\sqrt{3\pi}}
\end{equation*}
and
\begin{equation*}
C_{(a_1,|a-b|)}=\frac{243\cdot t_{k-2}}{4\sqrt{3\pi}}.
\end{equation*}
Suppose first that $3\nmid n+1$. Then we have either $a$ and $b$ are both odd or $a$ is odd and $b$ is even. By Lemma \ref{lem:5.7}, we have for all $n\geq\lfloor\frac{k}{3}\rfloor$
\begin{equation*}
A_{(a,b)}(n)=A_{(a_1,|a-b|)}(n)-A_{(a_0,a_1)}(n).
\end{equation*}
Then we have
\begin{align*}
A_{(a,b)}(n)&=\frac{C_{(a_1,|a-b|)}\cdot 6.75^n}{n^{3/2}}\left(1+\frac{s_{k-2}}{n}+O\left(\frac{1}{n^2}\right)\right)-\frac{C_{(a_0,a_1)}\cdot 6.75^n}{n^{3/2}}\left(1+\frac{s_{k-3}}{n}+O\left(\frac{1}{n^2}\right)\right)\\
&=\frac{(C_{(a_1,|a-b|)}-C_{(a_0,a_1)})\cdot 6.75^n}{n^{3/2}}\left(1+\left(\frac{t_{k-2}s_{k-2}}{t_k}-\frac{s_{k-3}t_{k-3}}{t_k}\right)\cdot\frac{1}{n}+O\left(\frac{1}{n^2}\right)\right)\\
&=\frac{(t_{k-2}-t_{k-3})\cdot 6.75^n}{4\sqrt{3\pi}n^{3/2}}\left(1+\left(\frac{t_{k-2}s_{k-2}}{t_k}-\frac{s_{k-3}t_{k-3}}{t_k}\right)\cdot\frac{1}{n}+O\left(\frac{1}{n^2}\right)\right).
\end{align*}
Thus $t_k-t_{k-2}=t_{k-3}$ and
\begin{equation*}
s_k=\frac{t_{k-2}s_{k-2}}{t_k}-\frac{s_{k-3}t_{k-3}}{t_k}.
\end{equation*}

The case when $3|n+1$ is similar.

This proves our claims. 

\end{document}